%% file: StabilitySUSYCompactifications.tex
\documentclass[a4paper,10pt,reqno]{amsart}
\usepackage[colorlinks=true, citecolor=black,linktoc=page]{hyperref}
\usepackage{amsthm,amsmath,amssymb}
\usepackage{color}
\usepackage{subcaption}

\allowdisplaybreaks

\title[Stability of spacetimes with SUSY compactifications]{Global stability of spacetimes with supersymmetric compactifications}
\date{\today}
\author[L. Andersson]{Lars Andersson}
\email{laan@aei.mpg.de}
\address{Albert Einstein Institute, Am M\"uhlenberg 1, D-14476 Potsdam, Germany }

\author[P. Blue]{Pieter Blue}
\email{p.blue@ed.ac.uk}
\address{Maxwell Institute and The University of Edinburgh, Peter Guthrie Tait Road, Edinburgh, EH9~3FD, UK}

\author[Z. Wyatt]{Zoe Wyatt}
\email{zoe.wyatt@ed.ac.uk}
\address{Maxwell Institute and The University of Edinburgh, Peter Guthrie Tait Road, Edinburgh, EH9~3FD, UK}

\author[S-T. Yau]{Shing-Tung Yau} 
\email{yau@math.harvard.edu}
\address{Center of Mathematical Sciences and Applications, Harvard University, 20~Garden Street,
Cambridge, MA 02138, USA}
\address{Department of Mathematics, Harvard University, Cambridge, MA 02138, USA }

\input{ExtraCommands.tex}
\begin{document}
\maketitle

\begin{abstract}
This paper proves the stability, with respect to the evolution determined by the vacuum Einstein equations, of the Cartesian product of high-dimensional Minkowski space with a compact, Ricci-flat Riemannian manifold that admits a spin structure and a nonzero parallel spinor. Such a product includes the example of Calabi-Yau and other special holonomy compactifications, which play a central role in supergravity and string theory. The stability proved in this paper provides a counter example to an instability argument by Penrose \cite{Penrose:InstabilityOfExtraSpaceDimensions}. 
\end{abstract}

\setcounter{tocdepth}{1}

\section{Introduction}\label{sec:introduction}
Let $(\Reals^{1+n}, \eta_{\Reals^{1+n}})$ be the $(1+n)$-dimensional Minkowski spacetime, and let $(\Compact,\cMet)$ be a compact, Ricci-flat Riemannian manifold that has a cover that admits a spin structure and a nonzero parallel spinor. The spacetime $\mathcal{M} = \Reals^{1+n} \times \Compact$ with metric 
\begin{align}
\bMet=\eta_{\Reals^{1+n}}+\cMet
\end{align}
is globally hyperbolic and Ricci flat, i.e, it is a solution to the $(1+n+d)$-dimensional vacuum Einstein equations. Such spacetimes play an essential role in supergravity and string theory \cite{CandelasHorowitzStromingerWitten}. In this paper we refer to $(\mathcal{M},\bMet)$ as a spacetime with a supersymmetric compactification and $(\Compact, \cMet)$ as the internal manifold. 

The simplest spacetime with a supersymmetric compactification is the Kaluza-Klein spacetime $(\Reals^{1+3}\times\Circle_\theta, \eta_{\Reals^{1+n}}+\di\theta^2)$, which has been studied since the 1920s \cite{Kaluza,Klein}. As shown by Witten in an influential paper  \cite{WittenKaluzaKlein}, this spacetime is unstable at the semiclassical level. Nonetheless in the same work Witten argued that the spacetime should be classically linearly stable. 

By contrast, Penrose has sketched an argument intended to show that spacetimes with supersymmetric compactifications are generically classically unstable, for every dimension $n$ and all internal manifolds, except possibly when the internal manifold is a flat $d$-dimensional torus \cite{PenroseBook, Penrose:InstabilityOfExtraSpaceDimensions}. There are theorems motivated by these considerations that generalize the classical singularity theorems to trapped surfaces of arbitrary co-dimension \cite{GallowaySenovilla,CiprianiSenovilla}. However, the results of the present paper show that for spacetimes with supersymmetric compactifications the instability argued by Penrose does not hold for $n\geq9$, and we conjecture here that in fact stability holds for $n\geq 3$.  The non-negativity of the spectrum of the Lichnerowicz Laplacian on symmetric 2-tensors, which holds for the internal spaces by the result of Dai, Wang, and Wei \cite{DaiWangWei}, plays a crucial role in our stability proof. In fact, this non-negativity, which is conjectured to hold for all compact Ricci flat manifolds, is sufficient for our result. See section \ref{sec:SpinorsLichnerowicz} for details.

In order to state our main theorem, we need to introduce some notation. For the product spacetime $\Reals^{1+n}\times\Compact$ we denote spacetime indices by $\ia,\mu,\nu\ldots$, Minkowski indices by $i,j,k \ldots$ and internal indices by $\ica,\icb,\icc\ldots$.  
For a general pseudo-Riemannian metric $g$, let $\nabla[g]$ denote its Levi-Civita connection, $\Riem[g]$ its Riemann curvature tensor, $\Ric[g]$ its Ricci curvature and $\di\mu_g$ its volume form. Define the following contraction
\begin{align}\label{def:Rcirc}
(R[g]\circ u)_{\mu\nu}&{}=R_{\mu\rho\nu\lambda}[g] u^{\rho\lambda} ,
\end{align}
which acts on symmetric $(0,2)$-tensors $u_{\mu\nu}$.
Given the supersymmetric spacetime metric $\bMet$ on $\Reals^{1+n}\times\Compact$, let
\begin{align}
(g_E)_{\mu\nu} = \bMet_{\mu\nu}+2(\di t)_\mu (\di t)_\nu .
\end{align}
where $\di t$ is with respect to the standard Cartesian coordinates on $\Reals^{1+n}$.
On $\Compact$ and $\Reals^{1+n}\times\Compact$ respectively, define the following inner products on $(0,2)$ tensors
\begin{align}
\langle u,v \rangle_\cMet={}&\cMet^{\ica\icc}\cMet^{\icb\icd} u_{\ica\icb}v_{\icc\icd}
 ,\\
\langle u,v \rangle_E ={}& g_E^{\mu\nu}g_E^{\rho\sigma} u_{\mu\rho}v_{\nu\sigma} .
\end{align}
Define $|u|_\cMet= (\langle u,u\rangle_\cMet)^{1/2}$, and similarly for $|u|_E$. 

The following is our main result. 
The details of some of the concepts appearing in the statement of the theorem appear in definitions \ref{def:initialDataSet}, \ref{def:solutionOfEinsteinEquationWithData}, \ref{def:WaveGauge}, \ref{def:higherDimensionalSchwarzschild} and theorem \ref{thm:higherDimensionalSchwarzschildExistsInWaveCoordinates}.

\begin{theorem}
\label{thm:MainResult}
Let $n,d\in\Integers^+$ be such that $n\geq9$, and let $\RegIndex\in\Integers^+$ be sufficiently large. Let $(\Reals^{1+n}\times\Compact,\bMet=\eta_{\Reals^{1+n}}+\cMet)$ be a spacetime with a supersymmetric compactification. 
Let $g_S$ denote the Schwarzschild metric in the $\eta_{\Reals^{1+n}}$-wave gauge with mass parameter $\SchwarzschildMass\geq 0$.

There is an $\epsilon>0$ such that if $(\Reals^n\times\Compact,\initialMetric,\initialIIFF)$ is an initial data set satisfying 
$\initialMetric=g_S+\cMet$ and $\initialIIFF=0$ where $|x|\geq 1$ and satisfying
\begin{align}\label{eq:MainThmSmallness}
\sum_{|I|\leq\RegIndex} \| \nabla[\initialMetric]^I (\initialMetric-\bMet|_{t=0})\|_{L^2(\Reals^n\times\Compact)}^2 
  +\sum_{|I|\leq\RegIndex-1} \|\nabla[\initialMetric]^I \initialIIFF\|_{L^2(\Reals^n\times\Compact)}^2 
+\SchwarzschildMass^2
\leq \epsilon,
\end{align}
then there is a solution $g$ of the vacuum Einstein equations on $\Reals^{1+n}\times\Compact$ with initial data $(\Reals^n\times\Compact,\initialMetric,\initialIIFF)$ and satisfying the $\bMet$-wave gauge. 
There is the bound 
\begin{align}
\sup_{(t,x^i,\omega)\in\Sigma_s\times\Compact}t^{2\decayRate}|g(t,x^i,\omega)-\bMet(t,x^i,\omega)|_E^2
&{}\lesssim \epsilon ,
\end{align}
where the decay rate is given by
\begin{align}\label{eq:DecayRate}
\decayRate ={}& \frac{n-2}{4} .
\end{align}
Finally $(\Reals^{1+n}\times\Compact,g)$ is globally hyperbolic and causally geodesically complete.
\end{theorem}

The stability result obtained in theorem \ref{thm:MainResult} covers a large class of product spacetimes, including many special holonomy compactifications relevant in supergravity and string theory. Although this paper succeeds in its goal of providing a counter example to the dimension-independent argument in \cite{Penrose:InstabilityOfExtraSpaceDimensions}, from a PDE perspective, theorem \ref{thm:MainResult} should be seen as a preliminary result, and we expect that the assumptions that $n\geq 9$, and that the Cauchy data is Schwarzschild near infinity can be relaxed. In fact we make the following conjecture. 

\begin{conjecture}
Spacetimes with a supersymmetric compactification and $n =3$ are nonlinearly stable.
\end{conjecture}

As explained below, this paper uses a relatively simple vector-field argument, while, for example, the proof of global stability for the coupled Einstein--Klein-Gordon system in $(1+3)$-dimensions \cite{LeFlochMa} has required combining vector-field arguments with estimates arising from control on the fundamental solution for the wave equation. 
Such detailed analysis is beyond the scope of this paper, but we intend to explore this in future work. Note that our current method can be easily used to show linear stability as far as $n=3$. 

The decay rate of $|h| \lesssim t^{-\decayRate}$ arises essentially as a linear estimate. The linearisation of the Einstein equation is 
\begin{align}
(\Box_\eta + \cLap + 2 R[\bMet]\circ) h_{\mu\nu} ={}& 0. 
\label{eq:linearisedEinstein}
\end{align}
To study conservation properties of the linear equations we introduce a novel stress-energy tensor
\begin{align}\label{eq:LinearStressEnergyTensor}
T[h]^\mu{}_\nu
&{}= \bMet^{\mu\ia}\langle\bCD_\ia h,\bCD_\nu h\rangle_E - \frac12\bMet^{\ia\ib} \langle\bCD_\ib h,\bCD_\ia h\rangle_E \delta^\mu_\nu\notag\\
&{}+ \langle R[\bMet]\circ h, h\rangle_E \delta^\mu_\nu,
\end{align}
which is specifically adapted to the tensorial operator appearing in \eqref{eq:linearisedEinstein}.
The conditions on $(\Compact,\cMet)$ imply, detailed further in section \ref{sec:SpinorsLichnerowicz}, that the energy integral derived from \eqref{eq:LinearStressEnergyTensor} is non-negative. 

The conditions on $(\Compact,\cMet)$ imply that the operator $-(\cLap+2R\circ)$ has a nonnegative discrete spectrum, and so a spectral decomposition can be applied to solutions $h$ of the linearised Einstein equation \eqref{eq:linearisedEinstein}. The spectral component corresponding to the zero eigenvalue satisfies an effective wave equation, $\Box_\eta (h^0)_{\mu\nu}=0$; the components corresponding to positive eigenvalues $\lambda$ satisfy effective Klein-Gordon equations $(\Box_\eta-\lambda) (h^\lambda)_{\mu\nu}=0$. A decomposition of this type has previously been used in the analysis of wave guides, where $\Compact$ is replaced by a compact subset of $\Reals^d$ with Neumann boundary conditions, see e.g.{} \cite{MetcalfeSoggeStewart,MetcalfeStewart}. When applying the vector-field method to the wave and Klein-Gordon equations, there is a unified approach using a basic energy of the form $\int \sum_{i=0}^n |\partial_i h|^2 +\lambda |h|^2 \diVol$ that can be strengthened by commuting the equation with $\Gamma$, the set of generators of translations, rotations, and boosts. 
 
This unified approach then bifurcates: the Klein-Gordon equation does not admit any further commuting first-order operators but the energy has a non-vanishing lower-order term  $\lambda|h|^2$; in contrast, the wave equation allows for commutation with the generator of dilations, $S=t\partial_t +r\partial_r$, but the lower-order term in the energy vanishes. For the quasilinear Einstein equation, we refrain from performing a spectral decomposition into wave and Klein-Gordon components. Thus, we use only the unified part of the approach (following especially the treatment of quasilinear Klein-Gordon equations in \cite{Hormander}), leaving us with a decay rate that is far from the sharp decay rates of the wave and Klein-Gordon equations. In particular, the vector-field method can be used to prove decay rates, for the wave and Klein-Gordon equations, of $t^{-(n-1)/2}$ and $t^{-n/2}$ respectively.

In light of this, it seems likely that some novel refinement should allow for a significantly better decay rate than $t^{-\decayRate}$ with $\decayRate=(n-2)/4$. 
This paper already contains two types of refinement. 
First, the decay rate is shown to be $s^{-2\decayRate}$ where $s^2=t^2-x^2$ inside light cones. The exponent $2\decayRate=(n-2)/2$ is much closer to the decay rate for the wave and Klein-Gordon equation. 
Second, the same decay rates are proved for $\Gamma^I h$ as for $h$, but, since the $\Gamma$ contain $t$- and $x$-dependent weights,  with respect to a translation invariant basis in Minkowski space, derivatives decay faster than the field $h$ itself. 

Having obtained a linear estimate that improves with increasing $n$, we take $n$ sufficiently large that $2\decayRate-2>1$, so that the nonlinear terms decay sufficiently fast for the linear estimates to remain valid. In particular, we take $n$ sufficiently large that we can ignore all nonlinear structure in the Einstein equation. It is well known that global existence results for semilinear equations in $(1+3)$-dimensions depend delicately on the nonlinearities, for example the null condition \cite{KlainermanNullLectures}. \cite{ChristodoulouKlainerman} used the vector-field method to prove the stability of Minkowski spacetime. One of the major advances in the simplified vector-field argument in \cite{LindbladRodnianski:MinkowskiStability,LindbladRodnianski:WeakNull,LindbladRodnianski:StabilityAgainstCompactPerturbations} was the introduction of the weak null condition and the observation that the Einstein equations in the harmonic gauge satisfy this condition. \cite{LeFlochMa} identified the relevant nonlinear structures for Klein-Gordon equations coupled to the $(1+3)$-dimensional Einstein equation.  

The dimension of the compact manifold only appears in the required regularity of the initial data, which is given explicity in theorem \ref{thm:ResultsReducedEquations}. The restriction to initial data which is exactly Schwarzschild outside of a compact set mirrors the proof of Minkowski stability in $(1+3)$-dimensions by \cite{LindbladRodnianski:StabilityAgainstCompactPerturbations}.

\subsection*{Background and Previous Work.}
Theories of higher-dimensional gravity are of great interest in supergravity and string theory as  possible models of quantum gravity. Many of these theories are built around the spacetimes with supersymmetric compactifications discussed above. 

Until now, the only nonlinear stability results have concerned the simplest Kaluza-Klein case when the internal space is the circle $\Circle$, or in slightly more generality, the flat $d$-dimensional torus. It was shown by one of the authors \cite{Wyatt} that this spacetime is classically stable to toroidal-independent perturbations. We remark that in the physics literature, these are known as zero-mode perturbations. An analagous result for cosmological Kaluza-Klein spacetimes, where the Minkowski spacetime is replaced by the 4-dimensional Milne spacetime, has also recently been shown \cite{BrandingFajmanKroncke}. 

The spacetimes of importance in supergravity and supergravity involve a nontrivial (i.e. non-toroidal) internal manifold with parallel spinors, such as a Calabi-Yau, $G_2$ or $Spin(7)$ manifold. Note that a solution of the 10 or 11-dimensional \textit{vacuum} Einstein equations can be considered as a particular solution of the supergravity equations. Local-in-time existence results are known for both the vacuum Einstein equations \cite{ChoquetBruhat52,ChoquetBruhatGeroch} and for the supergravity equations \cite{ChoquetBruhatSUGRA}. Furthermore, global-in-time existence and decay results for a nonlinear wave equation for $3$-form fields, on a fixed background spacetime with compact internal dimensions have been shown in \cite{Ettinger}. The field equation studied in \cite{Ettinger} is modelled on the supergravity equations with the gravitational interaction turned off. In our present work,  we  consider the stability of spacetimes with supersymmetric compactifications as solutions to the vacuum Einstein equations. In future work we intend to study their stability under the supergravity equations.

\subsection*{Outline of Paper.} 
In section \ref{sec:preliminaries} we introduce: the Lichnerowicz Laplacian, the foliation by hyperboloids, the gauge condition and the higher dimensional Schwarzschild -product spacetime. In section \ref{sec:Sobolev} we prove a Sobolev estimate on hyperboloids with respect to wave-like energies. In section \ref{sec:Energy} we define an energy functional adapted to the internal manifold and to hyperboloids. Finally in section \ref{sec:MainProof} we prove the main theorem.

\section{Preliminaries} \label{sec:preliminaries}
\subsection{Parallel Spinors and the Lichnerowicz Laplacian}\label{sec:SpinorsLichnerowicz}
Our main theorem has been stated for an internal manifold that has a cover that admits a spin structure and a nonzero parallel spinor. In this subsection we detail how this condition relates to a linear stability condition involving the eigenvalues of an operator closely related to the Lichnerowicz Laplacian. 

\begin{definition}[Riemannian Linear Stability]
\label{def:RiemannianLinearStable}
Define $\cLap=\cMet^{\ica\icb}\cCD_\ica\cCD_\icb$ to be the standard Laplacian on $(\Compact,\cMet)$. 
Let $u_{\ica\icb}$ be a symmetric $(0,2)$ tensor defined on $\Compact$. Define $\mL$ to act on such tensors by
\begin{align}
(\mL u)_{\ica\icb}={}& - \cLap u_{\ica\icb} - 2 (R[\cMet]\circ u)_{\ica\icb}.
\end{align}
We define a Ricci-flat manifold $(\Compact,\cMet)$ to be  Riemannian linearly stable iff
\begin{align}
\int_\Compact\langle \mL u, u\rangle_\cMet \diCVol
{}&\geq 0,
\label{eq:RiemannianLinearlyStable}
\end{align}
for all symmetric $(0,2)$-tensors $u_{\ica\icb}$.
\end{definition}

The operator $\mL$ is closely related to the Lichnerowicz Laplacian $\Lich$, which acts on symmetric tensors by
\begin{align}
(\Lich u)_{\ica\icb}={}&(\mL u)_{\ica\icb}+\Ric[\cMet]_{\ica\icc} u^\icc{}_\icb +\Ric[\cMet]^\icc{}_\icb u_{\ica\icc}.
\end{align}
Clearly on a Ricci-flat space these operators are equivalent. The operator $\mL$ is self-adjoint and elliptic, and consequently by the compactness of $\Compact$ and spectral theory, it has a discrete set of eigenvalues of finite multiplicity. Consequently the above definition \eqref{eq:RiemannianLinearlyStable} amounts to a condition $\lambda_{\min}\geq0$ on the lowest eigenvalue $\lambda_{\min}$ of $\mL$. For further details see e.g. \cite{Besse}. 

Our main theorem \ref{thm:MainResult} in fact applies more generally to internal manifolds which are Riemannian linearly stable. For the purposes of this paper, the crucial relation between spacetimes with a supersymmetric compactification and with an internal space that is Riemannian linear stable is the following. 

\begin{theorem}[{\cite[Theorem 1.1]{DaiWangWei}}]\label{thm:DaiWangWei}
If a compact, Ricci-flat Riemannian manifold $(\Compact,\cMet)$ has a cover which is spin and admits a nonzero parallel spinor then it is Riemannian linearly stable.
\end{theorem}

Note that some of the ideas established in \cite{DaiWangWei} date back to work of Wang  \cite{WangSpinors91} on the deformation theory of parallel and Killing spinors. A spin manifold $(\Compact,\cMet)$ with a non-zero parallel spinor is Ricci flat and has special holonomy, cf.  \cite{WangSpinors89} for a classification.
It is not known if any hypotheses on the internal space beyond Ricci flatness are necessary for stability to hold, as all known examples of compact Ricci-flat manifolds admit a spin cover with nonzero parallel spinors. The problem of constructing Ricci-flat manifolds including ones with non-special holonomy has been widely studied. A few relevant references on the topic are  \cite{MR3666569, MR1040196, MR1123371, MR3032330}. 

The spatial equivalent of the $\bMet$-wave gauge was used in the proof of Milne stability \cite{AnderssonMoncrief:Milne11}. This led to terms involving $\mL$ appearing in their PDEs, which were treated using Riemannian linear stability properties specific to the Milne spacetime.

\subsection{Cartesian, hyperbolic, and hyperbolic polar coordinates}

\begin{definition}[Minkowski space]
Let $n\geq1$ be an integer. Define Cartesian coordinates to be $(x^0,x^1,\ldots,x^n)$ $=(t,x^1,\ldots,x^n)$ $=(t,\vec{x})$ parameterising $\Reals^{1+n}$, and define 
\begin{align}
\eta_{\Reals^{1+n}} ={}& -\di t^2 +\sum_{i=1}^n (\di x^i)^2 . 
\end{align}
Define, for $i\in\{1,\ldots,n\}$, the translation vector fields $T$ and $X_i$ so that, in the Cartesian coordinates, they are given by
\begin{equation}\aligned
X_i ={}& \partial_{x^i}  ,\\
T ={} X_0 ={}& \partial_t .
\endaligned\end{equation}
Define, for $i,j\in\{0,\ldots,n\}$, the vector fields $Z_{ij}$ so that, in the Cartesian coordinates, they are given by
\begin{align}
Z_{ij}={}& (\eta_{\Reals^{1+n}})_{jk} x^k\partial_i -(\eta_{\Reals^{1+n}})_{ik}x^k\partial_j.
\end{align}
Define the collection of Lorentz generators by
\begin{align}
\LGen=\{ Z_{ij}, T, X_i\}.
\end{align}
Define $|x|^2=\sum_{i=1}^n(x^i)^2$ and define, in the region $t\geq|x|$, the hyperboloidal coordinates to be
\begin{equation}\aligned
s={}& (t^2-|x|^2)^{1/2} ,\\
y={}& x .
\endaligned \end{equation}
Define, for $i\in\{1,\ldots,n\}$,  the vector fields $Y_i$ so that, in the hyperboloidal coordinates, they are given by 
\begin{align}
Y_i ={}& \partial_{y^i} . 
\end{align}
For $s_0\geq 0$, define the spacelike hyperboloidal hypersurface 
\begin{align}
\Sigma_{s_0} = \{ (t,x)\in\Reals^{1+n}: t>0, s=s_0\}.
\end{align}
\end{definition}

Note that, because $(\eta_{\Reals^{1+n}})_{00}=-1$, $Z_{0i}={} t\partial_{x^i} +x_i\partial_t$. Furthermore the collection $\LGen$ is closed under commutation and forms a basis for the Poincar\'e Lie algebra. 

\begin{definition}[Spacetimes with a supersymmetric compactification]
On $\Reals^{1+n}\times\Compact$, define, for $i\in\{0,\ldots,n\}$, $X_i$ , $Y_i$, and $Z_{ij}$ to be as in $\Reals^{1+n}$. Let primed Roman letters denote spatial indices $i',j'\in\{1,\ldots,n+d+1\}$. 
Define the following collection of vector fields
\begin{align}
\Gamma
={}& Z \cup \{\cLap\} .
\end{align}
Note $[\LGen,\cLap]=0$.  
Define $\Naturals=\{0,1,2\ldots\}$. 
Define $\{Z_i\}_{i=1}^{(n+1)(n+2)/2}$ to be a reindexing of $\{X_i\}_{i=0}^n\cup\{Z_{ij}\}_{0\leq i<j\leq n}$, define a multi-index to be an ordered list of arbitrary length of elements from $\{1,\ldots,(n+1)(n+2)/2\}$, and for a multi-index $I=(i_1,\ldots,i_k)$ define the length $|I|=k$ and the differential operator $Z^I=Z_{i_k}\circ\ldots \circ Z_{i_1}$. 
For $I\in\Naturals$ and $u_{\mu\nu}$ a tensor defined on $\Reals^{1+n}\times\Compact$, define the following generalised multi-index notation
\begin{align*}
|\Gamma^I u|_E^2= \sum_{I_1:|I_1|+2j=|I|} |\LGen^{I_1}\cLap^j u|_E^2,
\end{align*}
where the sum is taken over all multi-indices $I_1
$ of length $|I_1|=k$ and integers $j$ such that $k+2j=|I|$. 
\end{definition}

\begin{definition}[Sobolev norms]
Let $u_{\mu\nu}$ be a tensor defined on $\Reals^{1+n}\times\Compact$ and $j\in\Naturals$. Define
\begin{align}
|\cCD^j u|^2_E=\cMet^{\ica_1\icb_1}\ldots\cMet^{\ica_j\icb_j} g_E^{\mu\nu} g_E^{\rho\sigma} (\cCD_{\ica_j}\ldots\cCD_{\ica_1} u_{\mu\rho}) (\cCD_{\icb_j}\ldots\cCD_{\icb_1} u_{\nu\sigma}).
\end{align} 
For $\ell\in\Naturals$ define the norms
\begin{align}
\|u(\cdot,\cdot,\omega)\|_{H^\ell(\Compact)}={}&\left(\int_\Compact \sum_{0\leq j\leq \ell}|\cCD^j u(\cdot,\cdot,\omega)|_E^2 \diCVol\right)^{1/2}, \\
\|u(t,x,\omega)\|_{L^2(\Sigma_s\times\Compact)}={}&\left(\int_{\Sigma_s\times\Compact} |u(t,x,\omega)|_E^2 \di x \diCVol\right)^{1/2},
\end{align}
where $\di x=\di x^1\ldots\di x^n$ is defined to be the flat Euclidean volume form.
\end{definition}

\begin{lemma}
\begin{align}
Y_i={}& X_i +\frac{x_i}{t} T ,\\
Z_{0i}={}& tY_i ,
\label{eq:ZExpandedAsY}\\
Z_{ij} ={}& y_i Y_j -y_j Y_i .
\end{align}
\end{lemma}

\begin{proof}
Since $t=\sqrt{s^2+y^2}$, by the chain rule, for $j\in\{1,\ldots,n\}$, 
$\frac{\partial}{\partial y^j}$
$=\frac{\partial x^i}{\partial y^j} \frac{\partial}{\partial x^i}$
$=\frac{\partial}{\partial x^j} 
+\frac{\partial t}{\partial y^j}\frac{\partial}{\partial t}$
$=\frac{\partial}{\partial x^j} 
+\frac{y_j}{t}\frac{\partial}{\partial t}$, which gives the first result. The second follows from multiplying both sides of the first by $t$. The third follows from 
$Z_{ij}$
$=x_iX_j -x_jX_i$
$=x_i(X_j+x_jt^{-1}T) -x_j(X_i+x_it^{-1}T)$. 
\end{proof}

The following two lemmas relate the $t$ coordinate to the $s$ coordinate. 

\begin{lemma}
Let $s\geq 1$. Suppose $(t_0,x_0)\in\Sigma_s$ and $(t,x)\in\Sigma_s$ with $|x-x_0|\leq t_0/2$. In this case, $t_0/2\leq t\leq 2t_0$. 
\end{lemma}

\begin{proof}
For the graph $t=\sqrt{s^2+|x|^2}$, the gradient 
$|\frac{\partial t}{\partial x}|$ $=|\frac{x}{\sqrt{s^2+|x|^2}}|$ $\leq 1$, so the change from $t$ to $t_0$ is less than the change in $|x|$ to $|x_0|$.
\end{proof}

\begin{lemma}
\label{lem:tBounds}
There is a constant $C>0$ such that for all $s>1$, in the portion of $\Sigma_s$ where $|x|\leq t-1$, one has $2t-1\leq s^2\leq t^2$.  
\end{lemma}

\begin{proof}
First, observe that $t^2=s^2+|x|^2\geq s^2$. Second, since $|x|^2\leq t^2-2t+1$, one has $s^2=t^2-|x|^2 \geq 2t-1 $. 
\end{proof}

The following are standard elliptic estimates, see for example \cite[\textsection Appx. H]{Besse}.

\begin{lemma}[Elliptic estimates on $(\Compact,\cMet)$]
\label{lem:CompactEllipticEst}
For $\ell\in\Naturals$ and $u_{\mu\nu}$ a sufficiently regular tensor defined on $\Reals^{1+n}\times\Compact$ there exist constants $c_1,c_2,c_3>0$ such that
\begin{align}
\| u \|_{H^{2\ell}(\Compact)} \leq c_1 \| (\cLap)^\ell u \|_{L^2(\Compact)} + c_2 \|u\|_{L^2(\Compact)}\leq c_3\| u \|_{H^{2\ell}(\Compact)}.
\end{align}
\end{lemma}

\subsection{The Einstein equations}
The theory of the Einstein equations is well known. In this section, we review this theory, for the sake of providing a self-contained presentation in this paper, and in particular to provide a self-contained statement of our main theorem \ref{thm:MainResult}.

\begin{definition}[Geometric initial data set]
\label{def:initialDataSet}
Let $m\in\Naturals^+$. An $m$-dimensional initial data set is defined to be a triple $(\Sigma,\initialMetric,\initialIIFF)$ such that $\Sigma$ is an $m$-dimensional manifold, $\initialMetric_{i'j'}$ is a Riemannian metric on $\Sigma$, $\initialIIFF_{i'j'}$ is a symmetric $2$-tensor on $\Sigma$, and the following equations (the constraint equations) are satisfied: 
\begin{equation}
R[\initialMetric] - |\initialIIFF|^2+(tr(\initialIIFF))^2=0, \quad
\nabla[\initialMetric]_{i'} tr(\initialIIFF)-\nabla[\initialMetric]^{j'} (\initialIIFF)_{i'j'} = 0 ,
\label{eq:Constraints}
\end{equation} 
where $tr(\initialIIFF)=\initialMetric^{i'j'}\initialIIFF_{i'j'}$.
\end{definition}

\begin{definition}[Solution of the Einstein equations with specified initial data]
\label{def:solutionOfEinsteinEquationWithData}
Let $\mathcal{M}$ be a manifold. A Lorentzian metric $g$ on $\mathcal{M}$ is defined to be a solution of the vacuum Einstein equations iff its Ricci curvature vanishes, 
\begin{align}
\label{eq:VacuumEinsteinEq}
\mathop{Ric}[g]_{\mu\nu}=0.  
\end{align}

Let $(\Sigma,\initialMetric,\initialIIFF)$ be a geometric initial data set. A solution to  the (geometric) Einstein equations with initial data $(\Sigma,\initialMetric,\initialIIFF)$ is defined to be a Lorentzian metric $g$ on $I\times\Sigma$ for some interval $I$ where one has:
$0\in I$, 
$g$ is a solution of the Einstein equations \eqref{eq:VacuumEinsteinEq}, 
$\{0\}\times \Sigma$ and $g$ restricted to vectors in $T(\{0\}\times \Sigma)$ is isometric in the category of Riemannian manifolds to $(\Sigma,\initialMetric)$, and, with the identification given by this isometry, the second fundamental form of the embedding of $\{0\}\times\Sigma$ into $I\times\Sigma$ is $\initialIIFF$. 
\end{definition}

In the previous definition, for convenience, we have required that the initial data be specified at $t=0$. This may initially appear more restrictive than definitions that are stated in other sources. However, because of the freedom to introduce new coordinate systems on the manifold $I\times\Sigma$, it is actually equivalent to definitions that allow initial data to specified at other values of $t$ or more general spacelike hypersurfaces.

\subsection{The reduced Einstein equations.}
\label{sec:ReducedEE}
To obtain a well-posed evolution problem for the Einstein equations we choose a gauge with respect to a fixed Lorentzian metric $\refMet_{\mu\nu}$ defined on $\mathcal{M}$. 

\begin{definition}[$\refMet$-wave gauge]
\label{def:WaveGauge}
For Lorentzian metrics $g$ and $\refMet$ defined on some manifold $\mathcal{M}$, let $\nabla[g]$ and $\refCD$ be the Levi-Civita connections with corresponding Christoffel symbols $\Gamma[g]$ and $\Gamma[\refMet]$ in local coordinates. Define the vector field $V^\ic$ in local coordinates by
\begin{align}
V^\ic = g^{\ia \ib}(\Gamma^\ic_{\ia\ib}[g]-\Gamma^\ic_{\ia\ib}[\refMet])\,.
\end{align}
Define also $V_\lambda = g_{\lambda\ib}V^\ib$.
The $\refMet$-wave gauge condition is given by 
\begin{align} \label{eq:WaveGauge}
V^\ic=0\,.
\end{align}
\end{definition}

Recall that the difference of two Christoffel symbols is a tensor, and so $V^\ic$ is in fact a well-defined vector field on $\mathcal{M}$.

\begin{definition}[The reduced Einstein equations]
Let $\mathcal{M}$ be a manifold with Lorentzian metric $\refMet$. A Lorentzian metric $g$ on $\mathcal{M}$ is defined to be a solution of the reduced Einstein equations iff
\begin{subequations} \label{eq:ReducedEinsteinEqs}
\begin{equation}\aligned
&g^{\ia\ib} \refCD_\ia \refCD_\ib g_{\mu \nu} - g^{\ic\id}\left( g_{\mu\lambda}\refMet^{\lambda\rho}\Riem[\refMet]_{\rho\ic\nu\id} + g_{\nu\lambda}\refMet^{\lambda\rho}\Riem[\refMet]_{\rho\ic\mu\id}\right) \\
&\qquad = Q_{\mu\nu}[g](\refCD g,\refCD g),
\endaligned\end{equation}
where we have defined
\begin{equation}\aligned\label{def:nonlinQ}
& Q_{\mu\nu}[g](\refCD g, \refCD g)=g^{\ic \id}g^{\ia\ib} \Big(\refCD_\nu g_{\id\ib} \refCD_\ia g_{\mu\ic} + \refCD_\mu g_{\ic\ia} \refCD_\ib g_{\nu\id} \\
&\quad -\frac12 \refCD_\nu g_{\id\ib} \refCD_\mu g_{\ic\ia} + \refCD_\ic g_{\mu\ia} \refCD_\id g_{\nu\ib} - \refCD_\ic g_{\mu\ia} \refCD_\ib g_{\nu\id} \Big).
\endaligned\end{equation}
\end{subequations}
\end{definition}

\subsection{The higher-dimensional Schwarzschild spacetime}\label{sec:HighDimSchwarz}

In this subsection, the higher-dimensional Schwarzschild solution is considered and its relationship to the initial data for the Einstein equations \eqref{eq:VacuumEinsteinEq} and the reduced Einstein equations \eqref{eq:ReducedEinsteinEqs} is discussed. The form of the metric is presented in the following definition. 

\begin{definition}
\label{def:higherDimensionalSchwarzschild}
Let $n\in\Integers$ be such that $n\geq 5$ and $\SchwarzschildMass\in[0,\infty)$. The Schwarzschild metric (in Schwarzschild coordinates) is defined for $(t,\bar{r},\omega)\in\Reals\times(\SchwarzschildMass^{1/(n-2)},\infty)\times S^{n-1}$ to be
\begin{align}
g_{S}={}&-\left(1-\frac{\SchwarzschildMass}{\bar{r}^{n-2}}\right)\di t^2
+\left(1-\frac{\SchwarzschildMass}{\bar{r}^{n-2}}\right)^{-1} \di \bar{r}^2 
+\bar{r}^2 \sigma_{S^{n-1}} .
\label{eq:higherDimensionalSchwarzschild}
\end{align}
\end{definition}

The above metric can also be written in the wave gauge. For $n=3$, it is sufficient to replace $(t,\bar{r},\omega)\in\Reals\times(\SchwarzschildMass^{1/(n-2)},\infty)\times S^{n-1}$ by $(t,x)=(t,r\omega)$ with $r=\bar{r}-M$; the resulting explicit metric can be found in \cite{LindbladRodnianski:StabilityAgainstCompactPerturbations,LeFlochMa}. Although the case $n=4$ leads to complicated terms involving logarithms, for $n\geq 5$, there is the following theorem. 

\begin{theorem}[{\cite[Section 5.2]{ChoquetBruhatChruscielLoizelet}}]
\label{thm:higherDimensionalSchwarzschildExistsInWaveCoordinates}
Let $n\in\Integers$ be such that $n\geq5$ and $\SchwarzschildMass\in[0,\infty)$. 
There are coordinates $(t,x)$ related to those in definition \ref{def:higherDimensionalSchwarzschild} by $(x^i)_{i=0}^n=(t,r(\bar{r})\omega)$ with 
\begin{align*}
r(\bar{r})
={} \bar{r}-\frac{\SchwarzschildMass}{2\bar{r}^{n-3}} +O(\bar{r}^{5-2n}),
\end{align*}
such that $(x^i)_{i=0}^n$ satisfy the harmonic gauge, that is, the $\eta_{\Reals^{1+n}}$-wave gauge. Furthermore, there exist functions $h_{00}(R)$, $h(R)$, and $\hat{h}(R)$, defined on an interval around $R=0$, that are analytic and bounded by a multiple of $C_S$ near $R=0$, and such that 
\begin{align}\label{eq:SchwarzschildHarmonicCoords}
g_{S}
={}&-\left(1 -\frac{h_{00}(r^{-1})}{r^{n-2}}\right)(\di x^0)^2
+\sum_{i,j=1}^n \left[\left(1 +\frac{h(r^{-1})}{r^{n-2}}\right)\delta^{ij}+\frac{\hat{h}(r^{-1})}{r^{n-2}}\frac{x^ix^j}{r^2}\right]\di x^i\di x^j .
\end{align}
In particular, the difference between the components of $g_{S}$ with respect to the harmonic coordinates and the corresponding components of the Minkowski metric are such that any $\partial^I$ derivative decays at least as fast as $C_S r^{-(n-2)-|I|}$. 
\end{theorem}

Note a result in \cite{Dai:SUSYPMT} ensures that $\SchwarzschildMass\geq0$ for the spacetimes of interest in our main theorem \ref{thm:MainResult}.

\section{Sobolev estimates on hyperboloids}
\label{sec:Sobolev}
We begin in lemma \ref{lem:HormanderWeightedSobolev} by recalling H\"ormander's proof of a Sobolev estimate on hyperboloids. This allows us  to introduce some of the key ideas that appear in our proof of the main result of this section, lemma \ref{lem:SobolevOnHyperboloidsInSUSY}. 

\begin{lemma}[Sobolev estimate for compactly supported functions on hyperboloids in Minkowski space {\cite[Lemma 7.6.1]{Hormander}}]
\label{lem:HormanderWeightedSobolev}
Let $\nu$ be the smallest integer greater than $n/2$ and $v\in C^\nu(\Reals^{1+n})$ have support in $|x|<t-1$. There is a constant $C$ such that
\begin{align}
\sup_{\Sigma_s} t^n |v(t,x)|^2 
\leq C \sum_{|I|\leq \nu} \int_{\Sigma_s} |\LGen^I v|^2 \di x . 
\end{align}
\end{lemma}

\begin{proof}
Consider a point $(t_0,x_0)\in \Sigma_s$ with $|x_0|^2\leq t_0^2-1$. Set $r_0=t_0/2$ and $y_0=x_0$. Set $\Sigma$ to be the portion of $\Sigma_s$ on which $|x-x_0|\leq r_0$. Let $(t,x)\in\Sigma$. This implies $|t-t_0|\leq r_0$, which implies $t/2\leq t_0\leq 2t$. Thus, 
\begin{align*}
\sum_{|I|\leq \nu} \int_{\Sigma_s} |\LGen^I v(t,x)|^2 \di x 
\geq{}& C\sum_{|I|\leq \nu} \int_{\Sigma_s} |t_0^{|I|}Y^I v(t,y)|^2 \di y .
\end{align*}
The right can be rewritten, by introducing rescaled coordinates $\tilde{y}=2t_0^{-1}(y-y_0)$ and $\tilde{v}(\tilde{y})=v(t,y)$. One can now decompose the portion of $\Sigma_s$ where $|x|\leq t-1$ into many subregions where $t$ does not vary by more than a factor of $2$. Let $\chi(\tilde{y})$ be a smooth cut-off such that $\chi$ is $1$ on a neighbourhood of $0$ and is $0$ for $|\tilde{y}|\geq 1/2$, it can further be bounded from below. A Sobolev estimate can then be applied to give a further lower bound on $v$. Combining these yields
\begin{align*}
\sum_{|I|\leq \nu} \int_{\Sigma_s} |t_0^{|I|}Y^I v(t,x)|^2 \di y 
={}& \sum_{|I|\leq \nu} \int_{|\tilde{y}|\leq 1} |\partial_{\tilde{y}}^I \tilde{v}(\tilde{y})|^2 t_0^n \di \tilde{y} \\
\geq{}& C t_0^n \sum_{|I|\leq \nu} \int_{|\tilde{y}|\leq 1} |\partial_{\tilde{y}}^I ((\chi \tilde{v})(\tilde{y}))|^2  \di \tilde{y} \\
\geq{}& C t_0^n |\tilde{v}(0)|^2\\
={}& C t_0^n |v(t_0,x_0)|^2 ,
\end{align*}
which completes the proof. 
\end{proof}

In the following lemma we obtain a Sobolev estimate for functions supported on product spacetimes with specified properties outside a compact set. In particular we obtain a pointwise estimate \eqref{eq:SobolevOnHyperboloidsInSUSYsDecay} in terms of the hyperboloidal time $s$, as well as a $t$-weighted pointwise estimate on a fixed hyperboloid \eqref{eq:SobolevOnHyperboloidsInSUSYtDecay}.

\begin{lemma}[Sobolev estimate for eventually prescribed functions on hyperboloids foliating product spacetimes]
\label{lem:SobolevOnHyperboloidsInSUSY}
Let $n\geq 4$, let $\dRegIndex$ be the smallest even integer larger than $d/2 $ and let $\tnu$ be the smallest integer greater than $ n/2 +\dRegIndex$. Let $u_{\mu\nu},f_{\mu\nu}$ be tensors on $\Reals^{1+n}\times\Compact$ with $f$ depending only the Minkowski coordinates $x^i$. Let $u\in C^{\tnu}(\Reals^{1+n}\times\Compact)$ satisfy $u=f$ for $|x|\geq t-1$. Let $f\in C^\infty(\Reals^{1+n}\times\Compact)$ be smooth and such that for all $I\in\Naturals$, there is a $C_I$ such that\footnote{The exponent on $f$ is set to match that corresponding to the exponent arising from the pointwise estimate \eqref{eq:SobolevOnHyperboloidsInSUSYsDecay} on $u$ in the region $|t-r|\leq C$. The limiting factor on the exponent in \eqref{eq:SobolevOnHyperboloidsInSUSYsDecay} arises from estimates on the hyperboloid, not from the decay of the prescribed function $f$. If a faster decay rate $t^{-\beta}$ could be proved (using similar methods) on hyperboloids for compact data, then a similar $t^{-\beta}$ decay could be proved for prescribed functions satisfying $f\leq r^{-\beta}$.} 
\begin{align}\label{eq:DecayAssumptionsSobolevInSUSY}
|\bCD^I f|_E\leq C_{|I|}|x|^{-(n-1)/2-|I|}.
\end{align}
Let $ \decayRate = \frac{n-2}{4}$. 
There is a constant $C$ such that,
\begin{equation}\aligned
\sup_{(t,x^i,\omega)\in\Sigma_s\times\Compact} s^{4\decayRate} |u(t,x^i,\omega)|_E^2 
\leq{}& C \sum_{|I|\leq \tnu} \sum_{i=1}^n  \int_{\underset{|x|\leq t-1}{\Sigma_s\times\Compact}} |Y_i \LGen^I u|_E^2 \di x \diCVol \\
{}&+ C\sum_{|I|\leq \tnu-1} C_I^2.  
\label{eq:SobolevOnHyperboloidsInSUSYsDecay}
\endaligned \end{equation}
Furthermore there is a constant C such that,
\begin{equation}\aligned
\sup_{(t,x^i,\omega)\in\Sigma_s\times\Compact} t^{2\decayRate} |u(t,x^i,\omega)|_E^2 
\leq{}& C \sum_{|I|\leq \tnu} \sum_{i=1}^n  \int_{\underset{|x|\leq t-1}{\Sigma_s\times\Compact}} |Y_i \LGen^I u|_E^2 \di x \diCVol \\
{}&+ C\sum_{|I|\leq \tnu-1} C_I^2.  
\label{eq:SobolevOnHyperboloidsInSUSYtDecay}
\endaligned \end{equation}
\end{lemma}

\begin{proof}
Lemma \ref{lem:CompactEllipticEst} and the standard Sobolev estimate imply
\begin{align*}
\sup_{\omega\in\Compact}|u(\cdot,\cdot,\omega)|_E \leq \|u\|_{H^{\dRegIndex}(\Compact)} \leq \|(\cLap)^{\dRegIndex/2} u\|_{L^2(\Compact)} + \|u\|_{L^2(\Compact)},
\end{align*}
for $\dRegIndex$ the smallest even integer greater than $d/2 $. This choice of $\dRegIndex$ being even is simply to make the  elliptic estimate cleaner. Note the trivial estimate
\begin{align*}
\sum_{|I|\leq \tnu-\dRegIndex}\left( |Y_i \LGen^{I}(\cLap)^{\dRegIndex/2}u|_E^2 + |Y_i \LGen^{I} u|_E^2\right)
\leq \sum_{|I|+2j\leq \tnu}|Y_i \LGen^{I}(\cLap)^{j}u|_E^2 \,.
\end{align*}
It is thus sufficient to prove in Minkowski space that
\begin{align} \label{eq:SobolevPartialGoal}
\sup_{\Sigma_s} s^{n-2} |u(t,x)|_E^2 
\leq{}& C \sum_{|I|\leq \tnu-\dRegIndex} \sum_{i=1}^n \int_{\Sigma_s} |Y_i \LGen^I u|_E^2 \di x 
+C \sum_{|I|\leq \tnu-1} C_I^2. 
\end{align}
since this would then imply
\begin{align*}
\sup_{\Sigma_s\times\Compact}s^{n-2}|u(t,x^i,\omega)|_E^2 
&{}\lesssim \sum_{|I|\leq \tnu-\dRegIndex} \sum_{i=1}^n \| \sup_\Compact (Y_i \LGen^I u ) \|_{L^2_x}^2+C \sum_{|I|\leq \tnu-1} C_I^2\\
&{}\lesssim \sum_{|I|\leq \tnu-\dRegIndex} \sum_{i=1}^n \| Y_i \LGen^I (\cLap)^{\dRegIndex/2}u \|_{L^2_x L^2_\Compact}^2 +C \sum_{|I|\leq \tnu-1} C_I^2. 
\end{align*}
For $|x|\geq t-1$ and $(t,x)\in \Sigma_s$, one has $t\sim|x|$, and so 
\begin{align*}
s^{n-2}|u(t,x)|_E^2 \leq t^{n-2}|u(t,x)|_E^2 \leq C |x|^{n-2} |u(t,x)|_E^2 \leq C |x|^{n-2} |f(x)|_E^2 \leq C C_0^2.
\end{align*} Thus, it remains to prove \eqref{eq:SobolevPartialGoal} for $|x|\leq t-1$. 

Consider the region $|x|\leq t-1$. Set $t_{\max}=(s^2+1)/2$, which is the value of $t$ at which $\Sigma_s$ intersects $|x|=t-1$ and which satisfies $t\leq t_{\max}\leq (t^2+1)/2$ on the portion of $\Sigma_s$ where $|x|\leq t-1$ by lemma \ref{lem:tBounds}.
Let $\chi:\Reals\rightarrow[0,1]$ be a smooth cut-off function such that $\chi(\ia)=1$ for $\alpha<1$ and $\chi(\alpha)=0$ for $\alpha>2$, and define the $(0,2)$ tensor $v_{\mu\nu}(t,x)=\chi(|x|/t_{\max})u_{\mu\nu}(t,x)$. Observe that $u_{\mu\nu}=v_{\mu\nu}$ in the region $|x|\leq t-1$. 

Hormander's proof of lemma \ref{lem:HormanderWeightedSobolev} relies on a carefully chosen rescaling of a portion of the hyperboloid, and the rest of this lemma follows the same idea, although the scaling is chosen differently.  
Recall both the Cartesian $(t,x)$ and hyperboloidal $(s,y)$ in Minkowski space, which are related via $(s,y)=(\sqrt{t^2-|x|^2},x)$. Given a choice of $s$, define $\tilde{y}=s^{-1}y$ and $\tilde{v}(\tilde{y})$ to be the value of $v$ at hyperboloidal coordinates $(s,s\tilde{y})$. With this $\di^n\tilde{y}=s^{-n}\di y$, $\partial_{\tilde{y}^i}=s\partial_{y^i}=s Y_i$. Recall $Z_i =t Y_i$. Thus, by a Sobolev estimate that exploits the fact that $1< n/2 < n/2+1$, 
\begin{align*}
\sup_{\Sigma_s} |v(t,x)|_E^2 
={}& \sup |\tilde{v}(\tilde{y})|_E^2 \\
\lesssim{}& \sum_{1\leq |J| \leq\frac{n}{2}+1} \int |\partial_{\tilde{y}}^J\tilde{v}|_E^2 \di^n\tilde{y} . 
\end{align*}
From rescaling and the facts that $s\leq t$ and that $Z_{0i}=tY_i$, it follows that
\begin{align*}
\sup_{\Sigma_s} |v(t,x)|_E^2 
\lesssim{}& s^{-n} \sum_{1\leq |J| \leq\frac{n}{2}+1} \int |(sY)^J v|_E^2 \di^n y \\
\lesssim{}& s^{-n+2} \sum_{0\leq |J| \leq\frac{n}{2}}\sum_i \int s^{2|J|}|Y^JY_i v|_E^2 \di^n y \\
\lesssim{}& s^{-n+2} \sum_{0\leq |J| \leq\frac{n}{2}}\sum_i \int t^{2|J|}|Y^JY_i v|_E^2 \di^n y \\
\lesssim{}& s^{-n+2} \sum_{0\leq |J| \leq\frac{n}{2}}\sum_i \int |Y_i \LGen^J v|_E^2 \di^n y . 
\end{align*}
The integral on the right can be decomposed into the parts where $|x|\leq t-1$ and $|x|>t-1$. Where $|x|\leq t-1$, the integral can be bounded by the integral term on the right-hand-side of \eqref{eq:SobolevPartialGoal} since $\tnu-\dRegIndex>n/2$.  Now consider the region $|x|>t-1$. Because of the support of $\chi$, it is sufficient to consider the region $t_{\max}-1\leq |x|\leq 2(t_{\max}-1)$. In this region, $v=\chi f$. When a derivative is applied to $v$, it is applied to either $\chi$ or to $f$, in which case one obtains an additional factor of $t_{\max}^{-1}$ or $|x|^{-1}$, from the properties of $\chi$ and $f$ respectively. Since $|x|/t_{\max}\in[1,2]$ in the support of $\partial \chi$, effectively, one obtains an extra factor of $|x|^{-1}$ in all cases, so $|Y_i\LGen^J v|_E\leq C C_{|J|+1}|x|^{-(n-1)/2-1}$, and 
\begin{align*}\int_{|x|\geq t_{\max}-1} |Y_i\LGen^J u|_E^2\di x
\leq{}& C C_{|J|+1}^2 \int_{\sphereN{n-1}}\int_{t_{\max}-1}^{2(t_{\max}-1)} (|r|^{-(n-1)/2-1})^2 |r|^{n-1}\di r \di^{n-1}\omega_{\sphereN{n-1}} \\
\leq{}& C C_{|J|+1}^2.
\end{align*} 

Observing that $s\geq Ct^{1/2}$ in the region $|x|\leq t-1$ allows us to obtain 
\begin{align*}
\sup_{\Sigma_s\times\Compact} t^{2\decayRate} |u|_E^2
&{}\leq \sup_{\Sigma_s\times\Compact\cap \{ |x|\leq t-1 \}} t^{2\decayRate} |u|_E^2 +\sup_{\Sigma_s\times\Compact\cap \{ |x|> t-1 \}} t^{2\decayRate} |u|_E^2\\
&{}\lesssim \sup_{\Sigma_s\times\Compact\cap \{ |x|\leq t-1 \}} s^{4\decayRate} |u|_E^2 + \sup_{\Sigma_s\times\Compact\cap \{ |x|> t-1 \}} r^{2\decayRate} |f|_E^2.\\
&{}\lesssim \sum_{|I|\leq \tnu} \sum_{i=1}^n  \int_{\underset{|x|\leq t-1}{\Sigma_s\times\Compact}} |Y_i \LGen^I u|_E^2 \di x \diCVol + \sum_{|I|\leq \tnu-1} C_I^2 \\
&{} + C_0 \sup_{\Sigma_s\times\Compact\cap \{ |x|> t-1 \}} r^{\frac{n-2}{2}} r^{-\frac{n-1}{2}}.
\end{align*}
In the final line we applied estimate \eqref{eq:SobolevOnHyperboloidsInSUSYsDecay} to the first term and assumption \eqref{eq:DecayAssumptionsSobolevInSUSY} to the second term. 
\end{proof}

\section{Energy integrals and inequalities}\label{sec:Energy}
\subsection{Basic properties of the energy}
The energy introduced in the following definition is related to  the standard energy used to study quasilinear hyperbolic PDEs, albeit with additional terms included in order to be compatible with the linearised equations \eqref{eq:linearisedEinstein}. 

\begin{definition}[Lichnerowicz-type energy on hyperboloids]\label{def:LichEnergyHyperboloids}
Let $n\in\Integers^+$ and let $\gamma^{\mu\nu},u_{\mu\nu}$ be tensors defined on $\Reals^{1+n}\times\Compact$. 
For $u,\gamma\in C^1(\Reals^{1+n}\times\Compact)$ and $s\geq 2$ define
\begin{align}
\mE[\gamma;u;s]
={}&\int_{\Sigma_s\times\Compact} \Big((s/t)^2|\partial_t u|_E^2 + \sum_{i=1}^n|Y_i u|_E^2 +  \langle \bCD^\ica u, \bCD_\ica u\rangle_E - 2 \langle R[\bMet]\circ u, u\rangle_E \notag\\
&\quad - 2\gamma^{\ia\ib}\langle\bCD_\ib u,\partial_t u\rangle_E n_\ia +\gamma^{\ia\ib}\langle\bCD_\ia u, \bCD_\ib u\rangle_E\Big) \di x \diCVol,
\end{align}
where $n_0 = 1,n_i=-x_i/t$ for $i\in\{1,\ldots,n\}$ and $n_\ica=0$, and $\di x$ is the flat Euclidean volume form.
\end{definition} 

Note that, following \cite{Hormander,LeFlochMa}, we have defined $\mE[\gamma;u;s]$ so that it is not the naturally induced energy associated with the metric $\bMet+\gamma$. This is because we have endowed $\Sigma_s$ with the flat Euclidean volume form $\di x$, instead of the induced Riemannian volume form $(s/t)\di x$.

The following lemma provides us with an energy functional which allows us to measure the perturbation of the spacetime. Note that in \eqref{eq:GammaHypDecay} we require some weighted $t-$decay on hyperboloids which we recover from \eqref{eq:SobolevOnHyperboloidsInSUSYtDecay} in lemma \ref{lem:SobolevOnHyperboloidsInSUSY}. 

\begin{lemma}[Basic properties of the energy] \label{lem:LichEnergyOnHyperboloids}
Take the conditions of definition \ref{def:LichEnergyHyperboloids}.
\begin{enumerate}
\item There is an $\epsilon_n>0$, such that if 
\begin{align} \label{eq:GammaHypDecay}
\sup_{\Sigma_s\times\Compact} t|\gamma|_E\leq C\epsilon_n,
\end{align} 
then for $s\geq 2$,
\begin{align}
\frac12 \mE[\gamma;u;s]
\leq \mE[0;u;s]
\leq 2 \mE[\gamma;u;s]. 
\label{eq:HyperboloidEnergyEquivalence}
\end{align}
\item \label{pt:EnergyInequality} If $u_{\mu\nu}$ is a solution of 
\begin{align}
(\bMet + \gamma)^{\ia \ib} \bCD_\ia \bCD_\ib u_{\mu\nu}+2(R[\bMet]\circ u)_{\mu\nu} = F_{\mu\nu},
\end{align}
then 
\begin{align}
& \mE[\gamma;u;s_1]
= \mE[\gamma;u;s_2]
+\int_{s_1}^{s_2} \int_{\Sigma_s\times\Compact} \langle F, \partial_t u \rangle_E \frac{s}{t} \di y \diCVol\di s 
\label{eq:LichEnergyOnHyperboloidsInequality}\\
&{}+\int_{s_1}^{s_2} \int_{\Sigma_s\times\Compact} \Big(-2(\bCD_\ia\gamma^{\ia\ib})\langle \bCD_\ib u,\partial_t u\rangle_E +(\partial_t\gamma^{\ia\ib})\langle\bCD_\ia u,\bCD_\ib u\rangle_E\Big) \frac{s}{t} \di y \diCVol\di s . \nonumber
\end{align}
\end{enumerate}
\end{lemma}

\begin{proof}
We first derive the energy $\mE[\gamma;u;s]$ by considering the following nonlinear version of the stress energy tensor \eqref{eq:LinearStressEnergyTensor} 
\begin{align}
T[\gamma;u]^\mu{}_\nu
&{}= (\bMet+\gamma)^{\mu\ia}\langle\bCD_\ia u,\bCD_\nu u\rangle_E - \frac12(\bMet+\gamma)^{\ia\ib} \langle\bCD_\ib u,\bCD_\ia u\rangle_E \delta^\mu_\nu\notag\\
&{}+ \langle R[\bMet]\circ u, u\rangle_E \delta^\mu_\nu.
\end{align}
We calculate
\begin{equation}\aligned
&{}\bCD_\mu T[\gamma;u]^\mu{}_\nu\\
&{}= \langle(\bMet+\gamma)^{\ia\ib}\bCD_\ia\bCD_\ib u,\bCD_\nu u\rangle_E+(\bMet+\gamma)^{\mu\ia}\langle\bCD_\ia u,\bCD_\mu \bCD_\nu u\rangle_E \\
&{}-(\bMet+\gamma)^{\ia\ib} \langle\bCD_\nu\bCD_\ib u,\bCD_\ia u\rangle_E +\bCD_\nu \langle R[\bMet]\circ u, u\rangle_E \\
&{}+(\bCD_\mu \gamma^{\mu\ia})\langle \bCD_\ia u, \bCD_\nu u\rangle_E - \frac12(\bCD_\nu\gamma^{\ia\ib})\langle\bCD_\ia u, \bCD_\ib u\rangle_E  .
\endaligned\end{equation}
Let $X^\mu$ be a vector field on $\Reals^{1+n}\times\Compact$ tangent to $\Reals^{1+n}$. We have
\begin{align*}
\bCD_\ia\bCD_\ib u_{\ic\id}=\bCD_\ib\bCD_\ia u_{\ic\id}+ \Riem[\bMet]_{\ia\ib\ic}{}^\rho u_{\rho\id} + \Riem[\bMet]_{\ia\ib\id}{}^\rho u_{\rho\ic}.
\end{align*}
However since $\Riem[\eta_{\Reals^{1+n}}]\equiv0$ we have
\begin{align}
\Riem[\bMet]_{\ia\ib\ic\id}X^\id = 0.
\end{align}
Consequently 
\begin{align*}
\langle \bCD_\ia\bCD_\ib u,\bCD_\nu u\rangle_E X^\ia=\langle \bCD_\ib\bCD_\ia u,\bCD_\nu u\rangle_E X^\ia.
\end{align*} 
and also
\begin{align*}
\bCD_\nu \langle R[\bMet]\circ u, u\rangle_E X^\nu=2\langle R[\bMet]\circ u, X^\nu\bCD_\nu u\rangle_E.
\end{align*}
This allows us to calculate
\begin{align*}
&{}\bCD_\mu (T[\gamma;u]^\mu{}_\nu X^\nu) \\
&{}= T^\mu{}_\nu [\gamma]\bCD_\mu X^\nu + \langle F, X^\nu\bCD_\nu u\rangle_E \\
&{}+(\bCD_\mu \gamma^{\mu\ia} ) \langle \bCD_\ia u, X^\nu\bCD_\nu u\rangle_E  - \frac12 (X^\nu \bCD_\nu\gamma^{\ia\ib})\langle\bCD_\ia u, \bCD_\ib u\rangle_E .
\end{align*}
Consider the hyperboloidal energy
\begin{align*}
\mE[\gamma;u;s]\\
={}&\int_{\Sigma_s\times\Compact} -2 T[\gamma;u]^\mu{}_\nu (\partial_t)^\nu n_\mu \di x\diCVol \\
={}& \int_{\Sigma_s\times\Compact} \Big(|\partial_t u|^2_E + \sum_{i=1}^n|\partial_i u|^2_E +\sum_{i=1}^n 2 \frac{x^i}{t} \langle \partial_t u, \partial_i u \rangle_E+ \cMet^{AB} \langle \bCD_A u, \bCD_B u\rangle_E \\
&-2\langle R[\bMet]\circ u, u \rangle_E - 2\gamma^{\mu\rho}\langle\bCD_\rho u,\partial_t u\rangle_E n_\mu +\gamma^{\rho\lambda}\langle\bCD_\rho u, \bCD_\lambda u\rangle_E\Big)\di x\diCVol.
 \end{align*}
where $n_0 = 1,n_i=-\eta_{ij}x^j/t$ for $i\in\{1,\ldots,n\}$ and $n_\ica=0$.
Note that 
\begin{equation}\aligned
\mE[0;u;s] = \int_{\Sigma_s\times\Compact} &{}\Big(|\partial_t u|^2_E + \sum_{i=1}^n|\partial_i u|^2_E + 2 \frac{x^i}{t} \langle \partial_t u, \partial_i u \rangle_E \\
&{}+ \langle \bCD^\ica u, \bCD_\ica u\rangle_E - 2 \langle R[\bMet]\circ u, u \rangle_E\Big)\di x\diCVol,
\endaligned\end{equation}
which alternatively can be written in hyperboloidal coordinates as
\begin{equation}\aligned
\mE[0;u;s] =\int_{\Sigma_s\times\Compact} &{}\Big(\left(s/t\right)^2|\partial_t u|^2_E +\sum_{i=1}^n|Y_{i} u|^2_E+ \langle \bCD^\ica u, \bCD_\ica u\rangle_E \\
&{}- 2 \langle R[\bMet]\circ u, u \rangle_E\Big)\di x\diCVol.
\endaligned\end{equation}
Since the contraction of $R[\bMet]$ with any direction tangent to $\Reals^{1+n}$ vanishes, and since $|w|_E\geq |w|_{\cMet}$ for any tensor field $w$, it follows from the definition of $\mL$ that
\begin{align*}
\int_\Compact&{}\Big(\langle \bCD^\ica u, \bCD_\ica u\rangle_E - 2 \langle R[\bMet]\circ u, u\rangle_E\Big) \diCVol\\
&{}\geq \int_\Compact \Big(\langle \bCD^\ica u, \bCD_\ica u\rangle_{\cMet} - 2 \langle R[\bMet]\circ u, u\rangle_{\cMet} \Big)\diCVol\\
&{}= \int_\Compact \langle \mL u, u\rangle_{\cMet} \diCVol .
\end{align*}
Thus, from theorem \ref{thm:DaiWangWei} and the condition of Riemannian linear stability \eqref{eq:RiemannianLinearlyStable}, it follows that 
\begin{align}
\int_\Compact&{}\Big(\langle \bCD^\ica u, \bCD_\ica u\rangle_E - 2 \langle R[\bMet]\circ u, u\rangle_E\Big) \diCVol
\geq 0.
\end{align}
This implies $\mE[0,u,s] \geq 0$. Using our previously calculated expression for the divergence of $T[\gamma;u]^\mu{}_\nu X^\nu$ we obtain via Stoke's theorem
\begin{align}
&{}\mE[\gamma;u;s_1]
= \mE[\gamma;u;s_2]
+\int_{s_1}^{s_2} \int_{\Sigma_s\times\Compact} \langle -2F, \partial_t u \rangle_E \frac{s}{t} \di y \diCVol\di s \\
&{}+\int_{s_1}^{s_2} \int_{\Sigma_s\times\Compact} \Big(-2(\bCD_\ia\gamma^{\ia\ib})\langle \bCD_\ib u,\partial_t u\rangle_E +(\partial_t\gamma^{\ia\ib})\langle\bCD_\ia u,\bCD_\ib u\rangle_E\Big) \frac{s}{t} \di y \diCVol\di s .\notag
\end{align}
This proves equality \eqref{eq:LichEnergyOnHyperboloidsInequality}.

Condition \eqref{eq:GammaHypDecay} combined with $s\geq Ct^{1/2}$ implies $\sup_{\Sigma_s\times\Compact}|\gamma|_E (t/s)^2 \leq C \varepsilon_n$. For simplicity denote $\cMet^{\ica\icb}\langle \bCD_\ica u,\bCD_\icb u\rangle_E$ by $|\partial_A u|_E^2$, then
\begin{align*}
\frac{s^2}{2t^2}(|\partial_t u|_E^2\sum_i+|\partial_i u|_E^2 + |\cCD u|_E^2 ) &{}\leq (|\partial_t u|^2+\sum_i|\partial_i u|^2 + |\cCD u|_E^2 ) (1-|x|/t) \\
&{}\leq |\partial_t u|_E^2+|\partial_i u|_E^2 +2\frac{x^i}{t}\langle \partial_t u,\partial_i u\rangle_E+ |\cCD u|_E^2 .
\end{align*}
Using this and Young's inequality we find
\begin{align*}
&{}|\mE[\gamma;u;s]-\mE[0;u;s]| \\
&{}= \left| \int_{\Sigma_s\times\Compact} \left( 2\gamma^{\ia\ib}\langle\bCD_\ia u,\partial_t u\rangle_E n_\ib -\gamma^{\ia\ib}\langle\bCD_\ia u, \bCD_\ib u\rangle_E\right)\di x\diCVol \right|\\
&{}\leq C \varepsilon_n \mE[0;u;s].
\end{align*}
and thus the energies are equivalent for sufficiently small $\varepsilon_n$. This proves estimate \eqref{eq:HyperboloidEnergyEquivalence}, completing the proof of the lemma. 
\end{proof}

Having defined the energy which involves first-order derivatives, we now introduce higher-order energies. 

\begin{definition}[Symmetry boosted energy]
Let $(\Reals^{1+n}\times\Compact,\bMet)$ be a spacetime with a supersymmetric compactification and $\RegIndex\in\Naturals$. 
For $k\leq\RegIndex$, define the energy of a symmetric tensor field $g$ to be 
\begin{align}
\mE_{k+1}(s)&{}=\sum_{|I|\leq k} \mE[g^{-1}-\bMet^{-1};\Gamma^I g;s].
\end{align}
\end{definition}

We end this section with the following Hardy estimate on hyperboloids. The proof is standard, see for example \cite[Lemma 2.4]{LeFlochMa}.

\begin{lemma}[Hardy estimate on hyperboloids]\label{lem:HardyInequality}
Let $u_{\mu\nu}$ be a tensor defined on $\Reals^{1+n}$, then one has
\begin{align}
\| r^{-1} u \|_{L^2(\Sigma_s)}
\lesssim
\sum_{i=1}^n \| Y_i u \|_{L^2(\Sigma_s)}.
\end{align}
\end{lemma}

\subsection{Preliminary $L^2$ and $L^\infty$-estimates}\label{sec:PreliminaryL2Linfty}
In our nonlinear estimates we will need to estimate terms of the form
\begin{align}\label{eq:DerivsDistibutedExample}
\LGen^I(\cLap)^j(uv) =\sum_{\substack{|I_1|+|I_2|=|I|\\|J_1|+|J_2|= 2j}} \LGen^{I_1}\cCD^{J_1}u \cdot \LGen^{I_2}\cCD^{J_2} v.
\end{align}
In the following lemma we estimate terms which appear as factors in the right hand side of \eqref{eq:DerivsDistibutedExample} in $L^2$ by using the elliptic estimates of lemma \ref{lem:CompactEllipticEst} and the Hardy estimate of lemma \ref{lem:HardyInequality}. Note the use of elliptic estimates allows us to avoid commuting derivatives, such as $[\cCD,\cLap]$, which makes the argument shorter. 

\begin{lemma}[$L^2$ estimate for distributed derivatives]\label{lem:DistributedDerivativesL2Estimate}
Let $u_{\mu\nu}$ be a tensor defined on $\Reals^{1+n}\times\Compact$. Suppose $\RegIndex$ is even, $\ell\in\Naturals$ and $\ell\leq \RegIndex+1$, then
\begin{equation}\aligned\label{eq:DistributedDerivativesL2Estimate}
\sum_{|I|+|J|\leq \ell}  \|t^{-1}\LGen^{I}\cCD^{J} u\|_{L^2(\Sigma_s\times\Compact)} 
&{}\lesssim  \mE_{\RegIndex+1}(s)^{1/2}.
\endaligned \end{equation}
\end{lemma}

\begin{proof}
We prove the estimate by considering separately the cases of $|I|=0$ and $|I|\neq0$. Firstly take $|I|\geq 1$, suppose $|J|=2m$ where $m\in\Naturals$ and consider $|I|+|J|=\ell\leq\RegIndex+1$. Using the elliptic estimates of lemma \ref{lem:CompactEllipticEst} we find
\begin{align*}
\|t^{-1}\LGen^{I}\cCD^{J} u\|_{L^2(\Sigma_s\times\Compact)}
&{}\lesssim \|\|t^{-1}\LGen^Iu\|_{H^{2m}(\Compact)}\|_{L^2(\Sigma_s)}\\
&{}\lesssim \| t^{-1}\LGen^I (\cLap)^m u\|_{L^2(\Sigma_s\times\Compact)}+\|t^{-1}\LGen^I u\|_{L^2(\Sigma_s\times\Compact)}\\
&{}\lesssim \sum_{i=1}^n\|Y_i Z^{I-1}(\cLap)^m u\|_{L^2(\Sigma_s\times\Compact)}+\sum_{i=1}^n\|Y_i \LGen^{I-1} u\|_{L^2(\Sigma_s\times\Compact)}\\
&{}\lesssim \mE[0;\LGen^{I-1} (\cLap)^mu;s]^{1/2}+\mE[0;\LGen^{I-1} u;s]^{1/2}\\
&{}\lesssim \mE_{\ell}(s)^{1/2} .
\end{align*} 
Next take $|I|\geq 1$ and suppose $|J|=2m+1$ where $m\in\Naturals$. For $|I|+|J|=\ell\leq\RegIndex+1$, again using lemma \ref{lem:CompactEllipticEst}, we have
\begin{align*}
\|t^{-1}\LGen^{I}\cCD^{J} u\|_{L^2(\Sigma_s\times\Compact)}
\lesssim{}& \|\|t^{-1}\LGen^Iu\|_{H^{2m+1}(\Compact)}\|_{L^2(\Sigma_s)}\\
\lesssim{}&\sum_{i=1}^n \| Y_i \LGen^{I-1}u\|_{L^2(\Sigma_s\times\Compact)}+\sum_{i=1}^n \| Y_i \LGen^{I-1}(\cLap)^m u\|_{L^2(\Sigma_s\times\Compact)}\\
&{}+\| \cCD( \LGen^I (\cLap)^m u)\|_{L^2(\Sigma_s\times\Compact)}\\
\lesssim{}& \mE[0;\LGen^{I-1} u;s]^{1/2}+\mE[0;\LGen^{I-1} (\cLap)^m u;s]^{1/2}\\
&{}+\mE[0;\LGen^{I} (\cLap)^m u;s]^{1/2}\\
\lesssim{}& \mE_{\ell}(s)^{1/2} .
\end{align*}

We now turn to the case $|I|=0$. Again we split into the cases of $|J|$ being even and odd. Start with $|J|=2m$ for $m\in \Naturals$. Note that $\RegIndex$ is chosen even so that we have the strict inequality $2m<\RegIndex+1$. Applying the Hardy estimate from lemma \ref{lem:HardyInequality}, and recalling that $t\geq r$ on the hyperboloid, yields
\begin{align*}
\| t^{-1} \cCD^{J} u\|_{L^2(\Sigma_s\times\Compact)}
&{}\lesssim \left\|\|r^{-1}u\|_{H^{2m}(\Compact)} \right\|_{L^2(\Sigma_s)} \\
&{}\lesssim\|r^{-1}(\cLap)^m u\|_{L^2(\Sigma_s\times\Compact)}+\| r^{-1}u\|_{L^2(\Sigma_s\times\Compact)}\\
&{}\lesssim \sum_{i=1}^n \| Y_i (\cLap)^m u\|_{L^2(\Sigma_s\times\Compact)}
+\sum_{i=1}^n \| Y_i u\|_{L^2(\Sigma_s\times\Compact)}  \\
&{}\lesssim \mE[0,(\cLap)^m u;s]^{1/2}+\mE[0,u;s]^{1/2}\\
&{}\lesssim \mE_{\RegIndex+1}(s)^{1/2}.
\end{align*}
Finally we consider the case $|I|=0$ and $|J|=2m+1\leq\RegIndex+1$ for $m \in \Naturals$. Again using lemma \ref{lem:HardyInequality} we obtain
\begin{align*}
\| t^{-1} \cCD^{J} u\|_{L^2(\Sigma_s\times\Compact)}
\lesssim{}& \left\|\|r^{-1}u\|_{H^{2m+1}(\Compact)} \right\|_{L^2(\Sigma_s)} \\
\lesssim{}&\|r^{-1} \cCD (\cLap)^m u\|_{L^2(\Sigma_s\times\Compact)}+\|r^{-1}u\|_{L^2(\Sigma_s\times\Compact)}\\
\lesssim{}&\| \cCD (\cLap)^m u\|_{L^2(\Sigma_s\times\Compact)}
+\sum_{i=1}^n \|Y_i u\|_{L^2(\Sigma_s\times\Compact)}  \\
\lesssim{}&\mE[0,(\cLap)^m u;s]^{1/2}+\mE[0,u;s]^{1/2}\\
\lesssim{}& \mE_{|J|}(s)^{1/2} .
\end{align*}
Adding together the above estimates over all appropriate multi-indices gives the required result. 
\end{proof}

\begin{corollary}[$L^2$ estimate for eventually prescribed functions on hyperboloids foliating product spacetimes]
\label{corol:L2DistDerivsOnHyperboloidsInSUSY}
Let $n\geq 4$. Let $u_{\mu\nu},f_{\mu\nu}$ be tensors defined on $\Reals^{1+n}\times\Compact$  with $f$ depending only on the Minkowski coordinates. Suppose $u=f$ for $|x|\geq t-1$. Let $f\in C^\infty(\Reals^{1+n}\times\Compact)$ be smooth and such that for all $I\in\Naturals$, there is a $C_I$ such that\footnote{Note that decay assumption on $f$ is stronger here than the assumption \eqref{eq:DecayAssumptionsSobolevInSUSY} in lemma \ref{lem:SobolevOnHyperboloidsInSUSY}.}
\begin{align}\label{eq:DecayAssumptionL2Estimate}
|\bCD^I f|_E\leq C_{|I|}|x|^{-(n+1)/2-|I|}.
\end{align}
Suppose $\RegIndex$ is even, $\ell\in\Naturals$ and $\ell\leq \RegIndex+1$, then
\begin{equation}\aligned\label{eq:WeightedDistributedDerivativesL2Estimate}
\sum_{|I|+|J|\leq \ell}\|(s/t)\LGen^{I}\cCD^{J} u\|_{L^2(\Sigma_s\times\Compact)}
&{}\lesssim  s \mE_{\RegIndex+1}(s)^{1/2} + \sum_{|I|+|J|\leq\ell}  C_{|I|,|J|}.
\endaligned \end{equation}
\end{corollary}

\begin{proof}
We will consider separately the regions $|x|\leq t-1$ and $|x|>t-1$. The estimate in the region $|x|\leq t-1$ follows by applying Lemma \ref{lem:DistributedDerivativesL2Estimate} with an additional factor of $s$. 
Next consider the region $|x|>t-1\geq t_0-1$ where we let $t_0=(s^2+1)/2$ be the value of $t$ at which $\Sigma_s$ intersects $|x|=t-1$. Using assumption \eqref{eq:DecayAssumptionL2Estimate} we find 
\begin{align*}
&{} \| (s/t)\LGen^{I}\cCD^{J} u\|_{L^2(\Sigma_s\times\Compact\cap\{|x|>t-1\}) }^2 \\
\leq{}& \int_{\Sigma_s\times\Compact\cap\{|x|>t_0-1\}}|\LGen^{I}\cCD^{J}u|_E^2\di x \di \mu_\cMet \\
\leq{}& C \int_{\Sigma_s\cap\{|x|>t_0-1\}} |\LGen^{I}\cCD^{J}f|_E^2\di x \\
\leq{}& C C_{|I|,|J|}^2 \int_{\sphereN{n-1}}\int_{\Sigma_s \cap\{|x|\geq t_0-1\}}(|r|^{-(n+1)/2})^2 |r|^{n-1}\di r \di\omega_{\sphereN{n-1}} \\
\leq{}& C C_{|I|,|J|}^2 \int_{\sphereN{n-1}}\int_{\Sigma_s \cap\{|x|\geq t_{0}-1\}}r^{-2} \di r \di\omega_{\sphereN{n-1}} \\
\leq{}& C C_{|I|,|J|}^2.
\end{align*} 
Adding together the above estimate over all appropriate multi-indices yields \eqref{eq:WeightedDistributedDerivativesL2Estimate}.
\end{proof}

We next use lemma \ref{lem:SobolevOnHyperboloidsInSUSY} to obtain $L^\infty$ estimates for terms which appear as factors in the right hand side of \eqref{eq:DerivsDistibutedExample}.

\begin{corollary}[Higher-order Sobolev estimates]\label{cor:SobolevHighDerivativesOnHyperboloidsInSUSY}
Let $n\geq 7$. Let $\dRegIndex, \tnu, u_{\mu\nu},f_{\mu\nu}$ be as defined in lemma \ref{lem:SobolevOnHyperboloidsInSUSY}. 
Then for $|I|+|J|=\ell \in \Naturals$ there is a constant $C$ such that 
\begin{equation} \aligned \label{eq:SobolevHighDerivativesOnHyperboloidsInSUSY}
&{}\sup_{\Sigma_s\times\Compact} \Big( s^{4\decayRate}|\LGen^{I} \cCD^{J} u|_E^2 + s^{4\decayRate-2}|(t/s) \LGen^{I} \cCD^{J} u|_E^2 \Big) \\
&{}\leq  C\sum_{|I|+2j\leq \tnu+\ell+1} \mE[0;\LGen^I (\cLap)^j u;s]+ C \sum_{|I|\leq \tnu+\ell-1} C_{|I|}^2.
\endaligned \end{equation}
\end{corollary}

\begin{proof}
We consider the left most term in \eqref{eq:SobolevHighDerivativesOnHyperboloidsInSUSY} first. 
Let $\tilde{\jmath}$ be the smallest even integer such that $\tilde{\jmath}\geq |J|$. In particular this means $|I|+|J|\leq |I|+\tilde{\jmath}\leq \ell+1$.  Recall that $\dRegIndex$ is the smallest even integer larger than $d/2 $ and $\tnu$ is the smallest integer greater than $ n/2 +\dRegIndex$.
Applying lemma \ref{lem:CompactEllipticEst} yields
\begin{align*}
\sup_\Compact |\cCD^{J} u|_E
\leq \|u\|_{H^{\dRegIndex+\tilde{\jmath}}(\Compact)}
\leq \|(\cLap)^{(\dRegIndex+\tilde{\jmath})/2} u\|_{L^2(\Compact)} + \|u\|_{L^2(\Compact)} .
\end{align*}
Thus, using in particular \eqref{eq:SobolevPartialGoal}, we have
\begin{align*}
&{}\sup_{(t,x,\omega)\in\Sigma_s\times\Compact}s^{4\decayRate}|\LGen^{I}  \cCD^{J} u(t,x^i,\omega)|_E^2 \\
\lesssim{}& \sum_{|I_1|\leq \tnu-\dRegIndex} \sum_{i=1}^n \| \sup_\Compact (Y_i \LGen^{I_1} \LGen^{I}\cCD^{J} u ) \|_{L^2(\Sigma_s)}^2+\sum_{|I_1|\leq \tnu-1} C_{I_1}^2\\
\lesssim{}& \sum_{|I_1|\leq \tnu-\dRegIndex} \sum_{i=1}^n \left( \| Y_i \LGen^{I+I_1} u \|_{L^2(\Sigma_s\times\Compact)}^2+ \| Y_i \LGen^{I+I_1} (\cLap)^{(\dRegIndex+\tilde{\jmath})/2}u \|_{L^2(\Sigma_s\times\Compact))}^2\right) \\
&{}+C\sum_{|I_1|\leq \tnu-1} C_{I_1}^2 \\
\lesssim{}& \sum_{|I|+2j\leq \tnu+\ell+1} \mE[0;\LGen^{I} (\cLap)^j u;s]+C\sum_{|I|\leq \tnu-1} C_{I}^2. 
\end{align*}

To complete the proof for the second term of \eqref{eq:SobolevHighDerivativesOnHyperboloidsInSUSY} we observe that $s\geq Ct^{1/2}$ in the region $|x|\leq t-1$ while we only have $s\leq t\leq r$ in the region $|x|>t-1$. Since $n\geq7$ we have $\decayRate\geq 1$ and thus
\begin{align*}
\sup_{\Sigma_s\times\Compact} s^{4\decayRate-2}|(t/s)Z^I\cCD^{J} u|_E^2
&{}\lesssim \sup_{\Sigma_s\times\Compact\cap \{ |x|\leq t-1 \}} (t^2/s^4) s^{4\decayRate} |Z^I\cCD^{J} u|_E^2 \\
&{}+ \sup_{\Sigma_s\times\Compact\cap \{ |x|> t-1 \}} s^{4\decayRate-4} r^2 |Z^I\cCD^{J} f|_E^2\\
&{}\lesssim \sum_{|I|+2j\leq \tnu+\ell+1} \mE[0;\LGen^I (\cLap)^j u;s]+ \sum_{|I|\leq \tnu+\ell-1} C_I^2\\
&{}+ C_I^2 \sup_{\Sigma_s\times\Compact\cap \{ |x|> t-1 \}}  r^{(n-2)-2} r^{-(n-1)}.
\end{align*}
Note in the final line we applied \eqref{eq:DecayAssumptionsSobolevInSUSY} and the first estimate of \eqref{eq:SobolevHighDerivativesOnHyperboloidsInSUSY}.
\end{proof}

\section{Proof of stability}\label{sec:MainProof}
\subsection{Stability for the reduced Einstein equations}
We now restate our main theorem \ref{thm:MainResult} in terms of the reduced Einstein equations. For convenience we translate the initial data of theorem \ref{thm:MainResult} to $\{t=4\}$. 

\begin{theorem}[Stability for the reduced Einstein equations]
\label{thm:ResultsReducedEquations}
Let $n,d\in\Integers^+$ be such that $n\geq9$  and let $\RegIndex\in\Naturals$ be an even integer strictly larger than $(n+d+8)/2$. Let $(\Reals^{1+n}\times\Compact,\bMet=\eta_{\Reals^{1+n}}+\cMet)$ be a spacetime with a supersymmetric compactification. 

Let $(\{t=4\}\times\Reals^n\times\Compact,\RedinitialMetric,\RedinitialDeriv)$ be Cauchy data for the reduced Einstein equations \eqref{eq:ReducedEinsteinEqs}. Assume that, for $|x|\geq 1$ with respect to Minkowski coordinates on $\Reals^{1+n}$, $(\RedinitialMetric,\RedinitialDeriv)=(g_{S}+\cMet,0)$ where $g_S$ is the Schwarzschild metric in the $\eta_{\Reals^{1+n}}$-wave gauge with parameter $\SchwarzschildMass\in[0,\infty)$. 

There is an $\epsilon>0$ such that, if the initial data satisfies
\begin{align}\label{eq:ThmReducedEEsSmallness}
\sum_{|I|\leq\RegIndex} \| \nabla[\RedinitialMetric]^I (\RedinitialMetric-\bMet|_{t=4})\|_{L^2(\Reals^{n}\times\Compact)}^2 
  +\sum_{|I|\leq\RegIndex-1} \|\nabla[\RedinitialMetric]^I \RedinitialDeriv\|_{L^2(\Reals^{n}\times\Compact)}^2 
+\SchwarzschildMass^2
\leq \epsilon,
\end{align}
then there is a future global solution $g_{\mu\nu}$ of the reduced Einstein equations \eqref{eq:ReducedEinsteinEqs} with initial data $(h,\partial_t h)|_{t=4}=(\RedinitialMetric,\RedinitialDeriv)$. Furthermore, there is the bound
\begin{align}
\sup_{(t,x,\omega)\in\Sigma_s\times\Compact}s^{4\decayRate}|g(t,x^i,\omega)-\bMet(t,x^i,\omega)|_E^2
&{}\lesssim \epsilon,
\end{align}
where $\decayRate$ was defined in \eqref{eq:DecayRate}. 
\end{theorem}

\begin{proof}
Let the perturbation and inverse perturbation be denoted
\begin{align}
h_{\mu\nu} ={}& g_{\mu\nu}-\bMet_{\mu\nu}, \\
H^{\mu\nu} ={}& g^{\mu\nu}-\bMet^{\mu\nu}.
\end{align}
Since $g$ is a solution of the reduced Einstein equation \eqref{eq:ReducedEinsteinEqs}, it follows that
\begin{equation}\label{eq:ReducedEinsteinEqsInPerturbedMetic}
\begin{aligned}
&{}(\bMet^{\ia\ib}+H^{\ia\ib})\bCD_\ia \bCD_\ib h_{\mu\nu}+2(R[\bMet]\circ h)_{\mu\nu} \\
&{}= Q_{\mu\nu}[g](\bCD h, \bCD h) + F_{\mu\nu}(H,h),
\end{aligned}\end{equation}
where $Q_{\mu\nu}$ is defined in \eqref{def:nonlinQ} and $F_{\mu\nu}$ is defined by
\begin{align}
F_{\mu\nu}(H,h)&{}=H^{\ia\ib}\left( h_{\ia\id}\Riem[\bMet]^\id{}_{\mu\nu\ib}+h_{\ia\id}\Riem[\bMet]^\id{}_{\nu\mu\ib}\right)\notag\\
&{}+ H^{\ia\ib}\left( h_{\mu\id}\Riem[\bMet]^\id{}_{\ia\nu\ib}+h_{\nu\id}\Riem[\bMet]^\id{}_{\ia\mu\ib}\right).
\end{align}
By commuting the symmetries $\LGen^I(\cLap)^j$ through the system \eqref{eq:ReducedEinsteinEqsInPerturbedMetic} we obtain
\begin{align}\label{eq:CommutedReducedEinsteinEqs} 
(\bMet^{\ia \ib} +H^{\ia\ib})\bCD_\ia \bCD_\ib (\LGen^I(\cLap)^j h_{\mu\nu})-2(R[\bMet]\circ \LGen^I(\cLap)^j h)_{\mu\nu} = \sum_{i=1}^3 F^{i,I,j}_{\mu\nu},
\end{align}
where 
\begin{equation}\aligned
F^{1,I,j}_{\mu\nu}&{}=\LGen^I(\cLap)^j Q_{\mu\nu}[g](\bCD h, \bCD h) \,,\\
F^{2,I,j}_{\mu\nu}&{}=\LGen^I(\cLap)^j F_{\mu\nu}(H,h)\,, \\
F^{3,I,j}_{\mu\nu}&{}=[\LGen^I(\cLap)^j, H^{\ia\ib}\bCD_\ia \bCD_\ib]h_{\mu\nu}\,.
\endaligned\end{equation}
The symmetry boosted energy is given by 
\begin{align}
\mE_{k+1}(s)
&{}=\sum_{|I|+2j\leq k} \mE[H;\LGen^{I} (\cLap)^j g;s].
\end{align}
From lemma \ref{lem:LichEnergyOnHyperboloids} and the Cauchy-Schwarz inequality we obtain
\begin{equation}\aligned\label{eq:SumEnergyInequality}
\mE_{\RegIndex+1}(s')^{1/2}
&{}\leq  \mE_{\RegIndex+1}(4)^{1/2} \\
&{}+\sum_{|I|+2j\leq\RegIndex}  \int_{4}^{s'}\left( \int_{\Sigma_s\times\Compact} \Big( \sum_{i=1}^3 |F^{i,I,j}|_E^2 +|G^{I,j}|_E^2 \Big)\di y\diCVol\right)^{1/2} \di s ,
\endaligned\end{equation}
where the $G^{I,j}$ terms arise from applying $\LGen^I(\cLap)^ j$ to the terms involving $\bCD \gamma$ or $\partial_t \gamma$ on the right side of the energy equality \eqref{eq:LichEnergyOnHyperboloidsInequality}. In particular, these can be bounded by 
\begin{align} \label{eq:GIjEstimate}
|G^{I,j}|_E^2 \leq C |\bCD H|_E^2|\LGen^I(\cLap)^j \bCD h|_E^2.
\end{align} 

The reduced field equations \eqref{eq:ReducedEinsteinEqsInPerturbedMetic} are a system of quasilinear, quasidiagonal wave equations for the spacetime metric $h_{\mu\nu}$. The existence of unique local solutions emanating from Cauchy data is standard \cite[Theorem 4.6 Appendix III]{ChoquetBruhatOUP}.

The proof then follows a boot-strap argument (or continuous induction). The goal is to prove that there exists a constant $ C>0$ and $\epsilon>0$ such that: if $\mE_{\RegIndex+1}(4)+C_S<\epsilon$ and $\forall s:$ $\mE_{\RegIndex+1}(s)\leq C\epsilon$, then $\forall s:$ $\mE_{\RegIndex+1}(s)\leq \epsilon +C\epsilon^2$ and hence $\mE_{\RegIndex+1}(s)\leq C\epsilon/2$. Note that there is clearly no loss of generality in placing our initial data at $t=4$. 

We consider the integral term on the right-hand-side in \eqref{eq:SumEnergyInequality} as the sum of integrals over $\Sigma_s \cap \{|x|\leq t-1\}$ and over $\Sigma_s \cap \{|x|> t-1\}$. Our approach is that, for sufficiently small $\SchwarzschildMass$, in the latter exterior region the solution is identically the product of Schwarzschild with the internal manifold. Thus in the region $|x|\geq t-1$ the perturbation $h_{\mu\nu}$ is only nonzero on its Minkowski indices and on these indices it is identically Schwarzschild. We note that sufficiently small compactly supported initial data on $\{t=4\}\cap\{|x|\leq 1\}$ can be extended to compactly supported initial data on $\Sigma_4$ \cite[Chapter 39]{LeFlochMaBook}.

Recall from section \ref{sec:HighDimSchwarz} that the difference between components of the Minkowski metric and the Schwarzschild metric in wave coordinates decay as $\SchwarzschildMass r^{-n+2}$ and the Christoffel symbols decay as $\SchwarzschildMass r^{-n+1}$. Along a geodesic parametrised by $\lambda$, one has $\di^2 x^i/\di\lambda^2 =\Gamma^i_{jk}(\di x^j/\di\lambda)(\di x^k/\di\lambda)$. Since $\SchwarzschildMass r^{-n+1}$ is integrable in $r$, there are geodesics along which $t$ and $r$ grow linearly and the $(\di x^j/\di\lambda)$ approach constant values, not all of which are vanishing. In particular, $\di r/\di t$ asymptotically approaches a constant, and this constant is $1$ for null geodesics. The next-to-leading order term in the geodesic equation arises from the metric, so it is of the form $C r^{-n+2}$, which is again integrable. 
Furthermore, the smaller the mass $\SchwarzschildMass$ the sooner this asymptotic behaviour comes to dominate. In particular, if $\SchwarzschildMass$ is sufficiently small, then any causal curve launched from within $\Sigma_4 \cap \{|x|\leq t-2\}$ can never reach the region where $|x|\geq t-1$. Furthermore, by uniqueness of solutions to quasilinear wave equations, since the initial data on $\Sigma_4$ is identically Schwarzschild for $|x|>t-2$, the solution is identically Schwarzschild for $|x|>t-1$. In particular, when estimating the components of the solution to \eqref{eq:CommutedReducedEinsteinEqs}, we can use the Sobolev lemma \ref{lem:SobolevOnHyperboloidsInSUSY} and corollary \ref{corol:L2DistDerivsOnHyperboloidsInSUSY} on hyperboloids with eventually prescribed functions. (The conclusion of this paragraph is essentially proposition 2.3 of \cite{LeFlochMa}.)

The estimate \eqref{eq:GammaHypDecay} required by lemma \ref{lem:LichEnergyOnHyperboloids} is established by combining \eqref{eq:SobolevOnHyperboloidsInSUSYtDecay} with the bootstrap assumptions and noting that since $n\geq9$ we certainly have $\decayRate>1$. Similarly since $n\geq9$ the decay assumptions \eqref{eq:DecayAssumptionL2Estimate} in corollary \ref{corol:L2DistDerivsOnHyperboloidsInSUSY} and \eqref{eq:DecayAssumptionsSobolevInSUSY} in lemma \ref{lem:SobolevOnHyperboloidsInSUSY} are satisfied.

We are now in a position to apply the results from section \ref{sec:PreliminaryL2Linfty} to the nonlinearities in \eqref{eq:SumEnergyInequality}. 
In general we will distribute $(s/t)(t/s)=1$ across the terms and estimate high-derivative terms with a factor of $(s/t)$ using corollary \ref{corol:L2DistDerivsOnHyperboloidsInSUSY} and low-derivative terms with a factor of $(t/s)$ using corollary \ref{cor:SobolevHighDerivativesOnHyperboloidsInSUSY}. We begin by estimating the term $G^{I,j}$. Using \eqref{eq:GIjEstimate} we find
\begin{align}
&{}\sum_{|I|+2j\leq\RegIndex}\|G^{I,j}\|_{L^2(\Sigma_s\times\Compact) } \notag\\
&{}\lesssim \sum_{|I|+|J|\leq\RegIndex} \left(\int_{\Sigma_s\times\Compact} |(t/s)\bCD H|_E^2|(s/t)\LGen^I\cCD^{J} \bCD h|_E^2\di y \diCVol\right)^{1/2} \notag\\
&{}\leq \sup_{\Sigma_s\times\Compact} \left(\big|(t/s)\bCD h\big|_E\right) 
	\left(\int_{\Sigma_s\times\Compact} \big|(s/t) \LGen^I\cCD^{J} \bCD  h\big|_E^2 \di y \diCVol\right)^{1/2} \notag\\
&{}\lesssim \frac{1}{s^{2\decayRate-1}} \Big( \mE_{\tnu+3}(s)^{1/2} + \SchwarzschildMass \Big)\Big(s\mE_{\RegIndex+1}(s)^{1/2}+ \SchwarzschildMass \Big).
\end{align}

The term $F_{\mu\nu}^1$ involves the standard quadratic derivative nonlinearities of the Einstein equations. Their weak-null structure is of course not relevant here since the Minkowski dimension is taken so high. We first look at what type of terms are contained in $F_{\mu\nu}^1$:
\begin{equation}\aligned\label{eq:F1DecompDeriv}
&{}\sum_{|I|+2j\leq\RegIndex}\|F^{1,I,j}_{\mu\nu}\|_{L^2(\Sigma_s\times\Compact) }\\
&{}\lesssim \sum_{|I|+|J|\leq\RegIndex} \left(\int_{\Sigma_s\times\Compact}
	\big|(\bMet+H)^{-1}\big|_E^2\big|\LGen^{I}\cCD^{J}(\bCD h \bCD h)\big|_E^2 \di y\diCVol \right)^{1/2} \\
&{}+ \sum_{\substack{|I_i|+|J_i|\leq\RegIndex\\|I_1|+|J_1|\geq 1}} \left(\int_{\Sigma_s\times\Compact}
	\big|\LGen^{I_1}\cCD^{J_1}h\big|_E^2\big|\LGen^{I_2}\cCD^{J_2}(\bCD h \bCD h)\big|_E^2 \di y\diCVol \right)^{1/2}.
\endaligned \end{equation}
We treat the first term on the right hand side of \eqref{eq:F1DecompDeriv} since the second term is higher-order and thus easier to estimate. Once again we estimate high-derivative terms with a factor of $(s/t)$ using corollary \ref{corol:L2DistDerivsOnHyperboloidsInSUSY} and low-derivative terms with a factor of $(t/s)$ using corollary \ref{cor:SobolevHighDerivativesOnHyperboloidsInSUSY}. This yields,
\begin{align}
&{}\sum_{\substack{|I|+|J|\leq\RegIndex}} \left(\int_{\Sigma_s\times\Compact}
	\big|(\bMet+H)^{-1}\big|_E^2\big|\LGen^{I}\cCD^{J}(\bCD h \bCD h)\big|_E^2 \di y\diCVol \right)^{1/2} \notag\\
&{}\lesssim \sum_{\substack{|I_i|+|J_i|\leq\RegIndex\\|I_2|+|J_2|\leq \frac{\RegIndex}{2}+1}} \left(\int_{\Sigma_s\times\Compact}
	C\big|\LGen^{I_1}\cCD^{J_1}\bCD h \big| \big|\LGen^{I_2}\cCD^{J_2} \bCD h)\big|_E^2 \di y\diCVol \right)^{1/2},
\end{align}
where by symmetry we can assume $|I_2|+|J_2|\leq \frac{\RegIndex}{2}+1$. After using $(s/t)(t/s)=1$ we find
\begin{align}
&{}\sum_{\substack{|I_i|+|J_i|\leq\RegIndex\\|I_2|+|J_2|\leq \frac{\RegIndex}{2}+1}} \left(\int_{\Sigma_s\times\Compact}
	C\big|(s/t)\LGen^{I_1}\cCD^{J_1}\bCD h \big| \big|(t/s)\LGen^{I_2}\cCD^{J_2} \bCD h)\big|_E^2 \di y\diCVol \right)^{1/2} \notag \\
&{}\lesssim \sup_{\Sigma_s\times\Compact} 
	\Big(\sum_{|I_2|+|J_2|\leq \frac{\RegIndex}{2}+1}
		|(t/s)\LGen^{I_2}\cCD^{J_2}\bCD  h|_E\Big)\notag\\
&{}	\times\sum_{|I_1|+|J_1|\leq\RegIndex}\left(\int_{\Sigma_s\times\Compact} 
	\left|(s/t) \LGen^{I_1}\cCD^{J_1}\bCD h\right|_E^2 \di y \diCVol\right)^{1/2} \notag\\
&{}\lesssim \frac{1}{s^{2\decayRate-1}} 
	\Big(\sum_{|I|+2j\leq \tnu+\frac{\RegIndex}{2}+3} 
	\mE[0;\LGen^I (\cLap)^j u;s]^{1/2}+ C_S\sum_{|I|\leq \tnu+\frac{\RegIndex}{2}} C_I^2\Big)
	\Big(s\mE_{\RegIndex+1}(s)^{1/2}+ \SchwarzschildMass\Big)\notag\\
&{}\lesssim \frac{1}{s^{2\decayRate-2}} 
	\Big( \mE_{\tnu+\frac{\RegIndex}{2}+4}(s)^{1/2} + \SchwarzschildMass\Big)
	\Big(\mE_{\RegIndex+1}(s)^{1/2}+ \SchwarzschildMass\Big). \label{eq:EstimateF1Explicit}
\end{align} 

The term $F_{\mu\nu}^2$ involves the new nonlinearities which are only nonzero when both $\mu,\nu\in\{\ica,\ldots,\icb\}$. This means we can control $F_{\mu\nu}^2$ by the following
\begin{align}
\sum_{|I|+2j\leq\RegIndex}\|F^{2,I,j}_{\mu\nu}\|_{L^2(\Sigma_s\times\Compact) } &{}\lesssim \sup_{\Sigma_s\times\Compact} \Big( \sum_{|I_0|\leq N}\big|\cCD^{I_0}\Riem[\cMet]\big|\Big) \notag \\
&{} \times \sum_{|I_i|+|J_i|\leq\RegIndex} \left(\int_{\Sigma_s\times\Compact} \big|\LGen^{I_1}\cCD^{J_1} h\big|_E^2 \big|\LGen^{I_2}\cCD^{J_2}h\big|_E^2 \di y\diCVol\right)^{1/2}. \label{eq:F2DistDeriv}
\end{align}
The Riemann curvature components of $\cMet$ are bounded (since $\Compact$ is compact) which allows us to control the first factor in \eqref{eq:F2DistDeriv}. To estimate the second factor in \eqref{eq:F2DistDeriv} we follow the same procedure as in $F^1_{\mu\nu}$, by controlling high-derivatives with a factor of $(s/t)$ using corollary \ref{corol:L2DistDerivsOnHyperboloidsInSUSY} and low-derivatives with a compensating factor of $(t/s)$ using corollary \ref{cor:SobolevHighDerivativesOnHyperboloidsInSUSY}. The result of this procedure leads to a term controlled by \eqref{eq:EstimateF1Explicit}. 

The final term $F_{\mu\nu}^3$ is a commutator involving the quasilinear perturbation of the principal part of the differential operator. Note first the identity
\begin{align}
\sum_{|I|+2j\leq\RegIndex}|F^{3,I,j}_{\mu\nu}|_E 
&{}\leq  C \sum_{\substack{|I_i|+|J_i|\leq\RegIndex\\|I_2|+|J_2|\leq \RegIndex-1}} |\LGen^{I_1}\cCD^{J_1} H|_E |\LGen^{I_1}\cCD^{J_1} \bCD \bCD h|_E.
\end{align}
Once again we distribute the product $(s/t)(t/s)=1$ across the two terms appearing here depending on where the derivatives land. The term with high-derivatives gains a factor of $(s/t)$ and is controlled using corollary \ref{corol:L2DistDerivsOnHyperboloidsInSUSY} while the term with low-derivatives absorbs a compensating factor of $(t/s)$ and is estimated using corollary \ref{cor:SobolevHighDerivativesOnHyperboloidsInSUSY}. Note that when the term $\LGen^{I_2}\cCD^{J_2} (\bCD \bCD h)$ is estimated in $L^\infty$ the Sobolev inequality will lead to a symmetry boosted energy at order $\tnu+\frac{\RegIndex}{2}+5$. We eventually obtain
\begin{align*}
&{}\sum_{|I|+2j\leq\RegIndex}\|F^{3,I,j}_{\mu\nu}\|_{L^2(\Sigma_s\times\Compact) }
\lesssim \frac{1}{s^{2\decayRate-2}} 
	\Big( \mE_{\tnu+\frac{\RegIndex}{2}+5}(s)^{1/2} + \SchwarzschildMass\Big)
	\Big(\mE_{\RegIndex+1}(s)^{1/2}+ \SchwarzschildMass\Big).
\end{align*}

Putting these all together, inserting the bootstrap assumptions and using also $\SchwarzschildMass^2 < \epsilon$  we find
\begin{equation}\aligned
\sum_{|I|+2j\leq\RegIndex}  \int_{4}^{s'}\Big( \int_{\Sigma_s\times\Compact} \Big( \sum_{i=1}^3 |F^{i,I,j}|_E^2 +|G^{I,j}|_E^2\Big)\di y\diCVol\Big)^{1/2} \di s
\lesssim \epsilon \int_4^{s'} \frac{1}{s^{2\decayRate-2}} \di s.
\endaligned\end{equation}
For integrability we require $2\decayRate-2> 1$, which is equivalent to each of the following
\begin{align}
\label{eq:IntegrabilityCondition}
\decayRate>{}&\frac32 , \\
n >{}& 8 . 
\end{align}
This implies $n \geq 9$. For the Sobolev estimates we require 
\begin{align}
\tnu + \RegIndex/2+4\leq \RegIndex.\end{align}
Recalling the definition of $\tnu$ given in lemma \ref{lem:SobolevOnHyperboloidsInSUSY} this certainly holds provided $\RegIndex> (n+d+8)/2$ and $\RegIndex$ is even. 

Consequently for sufficiently small $\epsilon$ and by Gr\"onwall's inequality applied to the energy estimate \eqref{eq:SumEnergyInequality} we find $\mE_{\nu+1}(s) \leq \tfrac12 C_1\epsilon$. We have thus obtained a future global solution $h_{\mu\nu}=g_{\mu\nu}-\bMet_{\mu\nu}$ to the reduced Einstein equations and which clearly satisfies the decay bounds given in theorem \ref{thm:ResultsReducedEquations}. 
\end{proof}

\begin{remark}
The system \eqref{eq:ReducedEinsteinEqsInPerturbedMetic} contains quadratic nonlinearities $F_{\ica\icb},F_{i\ica}$ that are new compared to the weak-null terms identified in the proof of Minkowski stability in  \cite{LindbladRodnianski:WeakNull, LindbladRodnianski:MinkowskiStability} and the proof of zero-mode Kaluza-Klein stability in \cite{Wyatt}. 
\end{remark}

\subsection{Proof of Theorem \ref{thm:MainResult}}
We are now in a position to use the results from theorem \ref{thm:ResultsReducedEquations} in order to prove our main result. 
Take an initial data set $(\Reals^n\times\Compact,\initialMetric,\initialIIFF)$ as specified in theorem \ref{thm:MainResult} with smallness conditions \eqref{eq:MainThmSmallness}. We now transform this data into the form required by theorem \ref{thm:ResultsReducedEquations}, which is a standard procedure, see for example \cite{LindbladRodnianski:StabilityAgainstCompactPerturbations}. We first set $((\RedinitialMetric)_{i'j'},(\RedinitialDeriv)_{i'j'})=(\initialMetric_{i'j'}, \initialIIFF_{i'j'})$.  Diffeomorphism invariance allow us the freedom to choose the lapse and shift. We set the shift to be zero $\shift_{i'}=0$. We choose the lapse to be a smooth function satisfying 
\begin{equation} \aligned
\lapse(r)&{}=1, \quad r\leq 1/2,\\
|\lapse-1|&{}\lesssim \SchwarzschildMass,\quad1/2\leq r\leq1,\\
\lapse(r)&{}=\left(1 -\frac{h_{00}(r^{-1})}{r^{n-2}}\right)^{1/2}, \quad r\geq 1. \\
\endaligned 
\end{equation}

We relate the lapse and shift with the Cauchy data for the reduced equations in theorem \ref{thm:ResultsReducedEquations}  by setting $(\RedinitialMetric)_{00}= -\lapse^2$ and
 $(\RedinitialMetric)_{0i'}=\shift_{i'}$. The initial data for $(\partial_t \lapse,\partial_t \shift_{i'})=((\RedinitialDeriv)_{00},(\RedinitialDeriv)_{0i'})$ is chosen by satisfying $V^\gamma=0$. This amounts to solving the following equations on $\Reals^n\times\Compact$
\begin{equation} \label{def:TimeDerivLapseShift}\aligned 
\lapse^{-3}\left((\RedinitialDeriv)_{00} +\lapse^2 \initialMetric^{i'j'}\initialIIFF_{i'j'}\right)&{}= \RedinitialMetric^{i'j'}\Gamma[\refMet]^0_{i'j'},\\
-\lapse^{-2}\initialMetric^{i'j'} (\RedinitialDeriv)_{0j'}-\lapse^{-1}\initialMetric^{i'j'}\partial_{j'} \lapse + \initialMetric^{j'k'}\Gamma^{i'}_{j'k'}[\initialMetric]&{}=\RedinitialMetric^{j'k'}\Gamma[\refMet]^{i'}_{j'k'}.
\endaligned \end{equation}
We have now brought the initial data of theorem \ref{thm:MainResult} into the form of theorem \ref{thm:ResultsReducedEquations}. It remains to check that our assumptions on the lapse and shift are compatible with smallness conditions \eqref{eq:ThmReducedEEsSmallness}.
To do this, recall the final sentence of theorem \ref{thm:higherDimensionalSchwarzschildExistsInWaveCoordinates}. 
This implies the following
\begin{align*}\int_{\{r\geq1\}\cap\Reals^n} |\nabla[\RedinitialMetric]^I (-\lapse^2-\eta_{00} )|^2\di x
\leq{}& \int_{\{r\geq1\}\cap\Reals^n}\SchwarzschildMass^2(r^{-(n-2)-|I|})^2 r^{n-1}\di r \di^{n-1}\omega_{\sphereN{n-1}} \\
\leq{}& \SchwarzschildMass^2 \int_{\{r\geq1\}\cap\Reals^n}r^{-(n-3)-2|I|} \di r \di^{n-1}\omega_{\sphereN{n-1}} \\
\leq{}& C \SchwarzschildMass^2.
\end{align*} 
By inverting the expressions \eqref{def:TimeDerivLapseShift} for $(\partial_t \lapse,\partial_t \shift_{i'})$ it is clear that the smallness conditions \eqref{eq:ThmReducedEEsSmallness} are satisfied.  
Furthermore it is a standard result, see for example \cite[Theorem 8.3]{ChoquetBruhatOUP},  that the future global solution constructed in theorem \ref{thm:ResultsReducedEquations} is in fact also a solution to the full Einstein equations.

Finally, note that the solution found in theorem \ref{thm:ResultsReducedEquations} is only defined to the future $t\geq4$. Nonetheless, by time translation, we can treat the initial data as being on $\{t=0\}$ instead of $\{t=4\}$, so that theorem 5.1 ensures the existence of a solution for $t\geq 0$. By time reversibility for the Einstein equation (and the reduced Einstein equation), we similarly obtain a solution for $t\leq 0$. Thus, we can construct the global solution required in theorem \ref{thm:MainResult}.

It now remains to prove the causal geodesic completeness of $(\Reals^{1+n}\times\Compact,g)$. 

Globally, the metrics $g$ and $\bMet$ are very close, in the sense that, with respect to a basis constructed from the $X_i$ and an orthonormal basis on $\Compact$, their components vanish to order $\epsilon$ globally. Denote from now onwards $T=\di t$. This is a globally timelike one-form such that $|g(T,T)-1| \lesssim \epsilon$. Thus, $g-2TT$ defines a Riemannian metric. (Note that in the introduction, we used the slightly different Euclidean metric $\bMet-2TT$.) Within this proof, we define, for a vector $u$, the Euclidean length to be
\begin{align}
|u|^2
={}& u^{\ia} u^{\ib} (g_{\ia\ib}+2T_{\ia}T_{\ib}).
\end{align}
Note that the fact that $g$ and $\bMet$ are very close implies the equivalence $|u|_E \sim |u|$. 

Consider a causal geodesic $\gamma$ that is affinely parameterised by $\lambda$. For the remainder of this paragraph, let $t=t(\lambda)$ denote the value of the Cartesian coordinate $t$ at the point $\gamma(\lambda)$. By rescaling, we may assume that $\di t/\di \lambda=1$ at $t=0$. Let $v$ be the (artificial, Euclidean) speed defined by $v\geq 0$ and 
\begin{align}
v^2 = \left| \frac{\di \gamma^\ia}{\di \lambda}\right|^2.
\end{align}
Since $g$ and $\bMet$ are very close, the rate of change in the $t$ direction cannot be (much) greater than the Euclidean speed, i.e.{} $|\frac{\di t}{\di\lambda}|$ $=|\frac{\di\gamma^0}{\di\lambda}|$ $\lesssim v$. On the other hand, since $\gamma$ is causal, the component of $\frac{\di\gamma}{\di\lambda}$ in the $T$ direction cannot vanish faster than the the length of the component in the orthogonal spatial directions, and the square of Euclidean velocity is the sum of the squares of the lengths of the $T$ components and the orthogonal spatial component (up to order $\epsilon$ multiplicative errors); thus $|\frac{\di t}{\di\lambda}|$ $=|\frac{\di\gamma^0}{\di\lambda}|$ $\gtrsim v$. In particular, there is the equivalence $|\frac{\di t}{\di\lambda}|$ $\sim v$.

The rate of change of the velocity is, since $\nabla[g] g=0$ and $\nabla[g]_{\frac{\di\gamma}{\di\lambda}}\frac{\di\gamma}{\di\lambda}=0$,
\begin{align}
\frac{\di}{\di\lambda} v^2
={}& 4\left(\frac{\di\gamma^\ia}{\di\lambda}T_\ia\right)\left(\frac{\di\gamma^\ib}{\di\lambda}\nabla[g]_{\frac{\di\gamma}{\di\lambda}} T_\ib\right) .
\end{align}
Since the absolute value of $\frac{\di\gamma^\ia}{\di\lambda}T_\ia=\frac{\di t}{\di \lambda}$ and the Euclidean length of $\frac{\di\gamma}{\di\lambda}$ are dominated by $v$
\begin{align}
\frac{\di v}{\di \lambda}
\lesssim{}& \left|\nabla[g]_{\frac{\di\gamma}{\di\lambda}} T\right| v . 
\end{align}
The $\CD[g] T$ can be expanded in terms of $g$ and $\CD[\bMet] g$. Both of these have norms that decay as $t^{-\decayRate}$ due to \eqref{eq:IntegrabilityCondition}. Thus, 
\begin{align}
\frac{\di v}{\di \lambda}
\lesssim{}& \epsilon t^{-\decayRate} v^2 . 
\end{align}
Thus, for $\epsilon$ sufficiently small, a simple bootstrap argument shows that $v\sim 1$ along all of $\gamma$. Thus, $\frac{\di t}{\di \lambda} \sim 1$. In particular, $t$ is monotone along $\gamma$. 

Let $t_{\sup}$ be the supremum of the $t$ values that are achieved along $\gamma$. For contradiction, suppose $t_{\sup}<\infty$. Since the length of the spatial component of $\frac{\di\gamma}{\di\lambda}$ is also uniformly equivalent to $v$, and hence to $\frac{\di t}{\di\lambda}$, it follows that, as $t\nearrow t_{\sup}$, the curve $\gamma$ has a limit in $\Reals^{1+n}\times\Compact$. Because of the global bounds on $g$ and its derivatives, by the standard Picard-Lindel\"of theorem for ODEs, the curve $\gamma$ must smoothly extend through this limiting point, contradicting the definition of $t_{\sup}$. Thus, $t_{\sup}=\infty$. The only other way in which $\gamma$ can be future incomplete is if $t$ diverges to $\infty$ in a finite $\lambda$ interval, but this is also impossible, since $\frac{\di t}{\di \lambda}\sim 1$. By time symmetry, the same argument holds in the past. 
Thus, any causal geodesic is complete. 

The previous construction shows that every causal geodesic goes through each level set of $t$. Thus, the level sets of $t$ are Cauchy surfaces, and $(\Reals^{1+n}\times\Compact,g)$ is globally hyperbolic.

\noindent\textbf{Acknowledgements.}
L.A. thanks the Mathematics Department at Harvard University for hospitality and support during the spring term of 2011, when the initial investigation of this topic was carried out. L.A. is also grateful to the YMSC at Tsinghua University, Beijing, for support and hospitality during part of the work of this paper, and thanks Pin Yu, and Po-Ning Chen for helpful discussions. The authors  L.A., P.B. and Z.W. thank the hospitality of the Institut Mittag-Leffler, Djursholm, Sweden during the winter semester of 2019. Z.W. is supported by MIGSAA, a CDT funded by EPSRC, the Scottish Funding Council, Heriot-Watt University and the University of Edinburgh. S-T.Y. is supported by the Black Hole Initiative, Award 61497, at Harvard University and by NSF Grant DMS-1607871 ``Analysis, Geometry and Mathematical Physics''.


\end{document}

%% file: ExtraCommands.tex
\newtheorem{theorem}{Theorem}
\newtheorem{definition}[theorem]{Definition}
\newtheorem{lemma}[theorem]{Lemma}
\newtheorem{corollary}[theorem]{Corollary}
\newtheorem{conjecture}[theorem]{Conjecture}

\newtheorem{remark}[theorem]{Remark}

\numberwithin{theorem}{section}

\newcounter{mnotecount}[section]

\newcommand{\Naturals}{\mathbb{N}}
\newcommand{\Integers}{\mathbb{Z}}
\newcommand{\Reals}{\mathbb{R}}

\newcommand{\Circle}{\mathbb{S}^1}
\newcommand{\sphereN}[1]{\mathbb{S}^{#1}}

\newcommand{\di}{\mathrm{d}}
\newcommand{\Ric}{\mathop{\mathrm{Ric}}}
\newcommand{\Riem}{\mathop{\mathrm{Riem}}}

\newcommand{\diVol}{\di\mu}

\newcommand{\Compact}{K}  
\newcommand{\cMet}{k}     
\newcommand{\bMet}{\hat{g}}  
\newcommand{\bCD}{\nabla[\bMet]}   
\newcommand{\refMet}{\hat{e}}  
\newcommand{\refCD}{\nabla[\refMet]} 
\newcommand{\mE}{\mathcal{E}}  
\newcommand{\LGen}{Z}
\newcommand{\lapse}{N}
\newcommand{\shift}{X}

\newcommand{\decayRate}{\delta(n)}  
\newcommand{\cLap}{\Delta_\Compact} 

\newcommand{\cCD}{{\CD[\cMet]}}   
\newcommand{\diCVol}{\di\mu_{\cMet}}  
\newcommand{\SchwarzschildMass}{C_S}
\newcommand{\initialMetric}{\gamma}
\newcommand{\initialIIFF}{\kappa}
\newcommand{\RedinitialMetric}{{g_0}}
\newcommand{\RedinitialDeriv}{{g_1}}

\newcommand{\CD}{\nabla} 

\newcommand{\Lich}{\Delta_L}
\newcommand{\mL}{\mathcal{L}}
\newcommand{\RegIndex}{N}
\newcommand{\dRegIndex}{\tilde d}
\newcommand{\tnu}{\tilde\nu}

\newcommand{\ia}{\alpha}
\newcommand{\ib}{\beta}
\newcommand{\ic}{\gamma} 
\newcommand{\id}{\delta}
\newcommand{\ica}{A}
\newcommand{\icb}{B}
\newcommand{\icc}{C}
\newcommand{\icd}{D}

\usepackage{tikz}
\usetikzlibrary{shapes,decorations,backgrounds}
\usetikzlibrary{shapes.multipart}
\usetikzlibrary{cd}
\usetikzlibrary{arrows} 
\usetikzlibrary{graphs}
\usetikzlibrary{intersections}
\usetikzlibrary{positioning}
\usetikzlibrary{decorations.pathmorphing}
\usetikzlibrary{datavisualization} 
\usetikzlibrary{datavisualization.formats.functions}
\usetikzlibrary{intersections}
\usetikzlibrary{hobby}
\usetikzlibrary{through}
\usetikzlibrary{calc}
\usetikzlibrary{fadings}
\usetikzlibrary{patterns}
\usepackage{pgfmath}
\usepackage{pgfkeys}

\usetikzlibrary{automata,positioning}
\usetikzlibrary{decorations.markings}
\tikzset{->-/.style={decoration={
  markings,
  mark=at position #1 with {\arrow{>}}},postaction={decorate}}}
  \tikzset{-<-/.style={decoration={
  markings,
  mark=at position #1 with {\arrowreversed[red]{latex'}}},postaction={decorate}}}

\usetikzlibrary{external}
\tikzstyle{singularity}=[red,dotted]
\tikzstyle{infinity}=[green,dashed]
\tikzstyle{regular}=[black]
\tikzset{dashdot/.style={dash pattern=on .4pt off 3pt on 4pt off 3pt}}
\tikzstyle{CauchyHorizon}=[orange,dashdot]